\documentclass[preprint,12pt]{elsarticle}

\usepackage{amscd}
\usepackage{amsmath}
\usepackage{amsthm}
\usepackage{amsfonts}
\usepackage{graphicx}
\usepackage{amssymb}
\usepackage{color}
\usepackage{amsbsy}
\usepackage{graphicx}
\usepackage{amssymb,color,amsbsy}
\usepackage{mathtools}
\usepackage{leftindex}

\usepackage{float}
\usepackage{comment}

\usepackage{stmaryrd}

\usepackage[ruled]{algorithm2e}

\usepackage[matrix,arrow]{xy}

\DeclareMathAlphabet{\mathpzc}{OT1}{pzc}{m}{it}

\usepackage{pgfpages}
\usepackage{tikz}
\usetikzlibrary{shapes.geometric}
\usetikzlibrary{decorations.pathmorphing}
\usetikzlibrary{arrows}
\usetikzlibrary{backgrounds}
\usetikzlibrary{positioning}
\usetikzlibrary{fit}
\usepackage{caption}
\usepackage{amsthm}

\usepackage{amsmath}

\usepackage[pagebackref=true, bookmarksopen=true,colorlinks=true, linkcolor=red,citecolor=blue]{hyperref}

\usepackage[capitalise]{cleveref}  

\global\long\def\ii{\cap}%

\global\long\def\U{\bigcup}%

\global\long\def\sm{\sum}%

\global\long\def\ñ{\sim}%

\global\long\def\Sa#1{\textnormal{Sachs}(#1)}%

\global\long\def\SSa#1{\textnormal{Sachs}^{\textnormal{prk}}(#1)}%

\global\long\def\prk#1{\textnormal{prk}(#1)}%

\global\long\def\perm#1{\textnormal{perm}(#1)}%

\global\long\def\a#1{\left|#1\right|}%

\usepackage{tikz-cd}

\usepackage{tkz-graph}

\usepackage{multicol}

\usepackage{subcaption}

\newtheorem{theorem}{Theorem}[section]

\newtheorem{corollary}[theorem]{Corollary}

\newtheorem{lemma}[theorem]{Lemma}

\newtheorem{problem}[theorem]{Problem}

\usepackage{comment}
\begin{document}

\begin{abstract}
	Several graph decompositions that factorize the determinant of the adjacency matrix isolate a K\H{o}nig-Egerv\'ary part, such as the SD--KE decomposition and the critical independence decomposition of Larson. This suggests that the study of graph unimodularity can be approached, to a large extent, through the structure of K\H{o}nig-Egerv\'ary graphs. In this paper we advance this point of view by introducing a new determinant factorization inside the class of K\H{o}nig-Egerv\'ary graphs. More precisely, given a K\H{o}nig-Egerv\'ary graph $G$, we consider the partition of $V(G)$ into its perfect-flower part $PF(G)$ and its perfect-flower-free part $PFF(G)$, and prove that
	\[
	\det(G)=\det(G[PF(G)])\det(G[PFF(G)]).
	\]
	We also obtain the analogous factorization for the permanent. This decomposition provides a new tool for the study of unimodularity, reducing the problem to two induced subgraphs of a very different nature: the graph $G[PF(G)]$, whose structure is closely related to Sterboul--Deming configurations with perfect matching, and the graph $G[PFF(G)]$, which is governed by the theory of critical independent sets. In this way, the paper gives a new structural framework for the study of unimodular graphs through K\H{o}nig-Egerv\'ary theory.
\end{abstract}

\begin{frontmatter} 
	\title{On the Determinant of K\H{o}nig--Egerváry Graphs}


\author[DEPTO]{Kevin Pereyra} 	\ead{kdpereyra@unsl.edu.ar, kevin.pereyra767@gmail.com}




	\address[DEPTO]{Departamento de Matem\'atica, Universidad Nacional de San Luis, San Luis, Argentina.} 	
	
	\date{Received: date / Accepted: date} 
	
\end{frontmatter} %

\section{Introduction}

Let $\alpha(G)$ denote the cardinality of a maximum independent set,
and let $\mu(G)$ be the size of a maximum matching in $G=(V,E)$.
It is known that $\alpha(G)+\mu(G)$ equals the order of $G$,
in which case $G$ is a König--Egerváry graph 
\cite{deming1979independence,gavril1977testing,stersoul1979characterization}.
K\H{o}nig-Egerv\'{a}ry graphs have been extensively studied
\cite{bourjolly2009node,jarden2017two,levit2006alpha,levit2012critical,jaume5404413kr}.
It is known that every bipartite graph is a König--Egerváry graph 
\cite{egervary1931combinatorial}. These graphs were independently introduced by Deming \cite{deming1979independence}, Sterboul \cite{stersoul1979characterization}, and \cite{gavril1977testing}.

The term subgraph in here is understood as a subgraph defined by a
graph and a given matching in that graph. In \cite{edmonds1965paths}, Edmonds introduced the following concepts
relative to a matching $M$ of a graph $G$ and its subgraphs. An
$M$-blossom of $G$ is an odd cycle of length $2k+1$ with $k$ edges
in $M$. The vertex not saturated by $M$ in the cycle is called
the \textit{base} of the blossom. An $M$-stem is an $M$-alternating
path of even length (possibly zero) connecting the base of the blossom
with a vertex not saturated by $M$ in $G$. The base is the only common vertex
between the blossom and the stem. An $M$-flower is a blossom joined
with a stem. The vertex not saturated by $M$ in the stem is called
the \textit{root} of the flower.

In \cite{stersoul1979characterization}, Sterboul introduced the concept
of a \textit{posy} for the first time. An $M$-posy consists of two (not necessarily disjoint) $M$-blossoms
	joined by an M-alternating path that starts and ends with edges in M. The endpoints of the path are the bases
	of the two blossoms. There are no internal vertices of the path in
	the blossoms. 

Sterboul \cite{stersoul1979characterization} was the first to characterize
König--Egerváry graphs via forbidden configurations relative to
a maximum matching. Subsequently, Korach, Nguyen,
and Peis \cite{korach2006subgraph} reformulated this characterization
in terms of simpler configurations, unifying the structures of
flowers and posies. Later, Bonomo et al. \cite{bonomo2013forbidden}
obtained a purely structural characterization based on forbidden subgraphs.
More recently, in \cite{jaume2025confpart3,jaume2025confpart2,jaume2025confpart1},
results were obtained that simplify working with flower and posy
structures.

\begin{theorem}[\cite{stersoul1979characterization}]\label{safe}
	For a graph $G$, the following properties are equivalent: 
	\begin{itemize}
		\item $G$ is a non-K\H{o}nig-Egerváry graph.
		\item For every maximum matching $M$, there exists an $M$-flower or an $M$-posy
		in $G$.
		\item For some maximum matching $M$, there exists an $M$-flower or an $M$-posy
		in $G$.
	\end{itemize}
\end{theorem}

The set of vertices of $G$ lying in a flower or posy, for any maximum matching,
is denoted by $SD(G)$, and we write $KE(G)=V(G)-SD(G)$.
The sets $SD(G)$ and $KE(G)$ constitute the SD--KE decomposition of the graph. 
A graph $G$ is called a Sterboul–Deming graph if $KE(G)=\emptyset$. Essentially, it can be regarded as the structural counterpart of a K\H{o}nig-Egerváry graph. Characterizations of Sterboul–Deming graphs can be found in \cite{kevinSDKECHAR}.

The SD--KE decomposition admits a factorization for the determinant of
the adjacency matrix of a graph \cite{jaume2025confpart1,KEVINnota}, that is,
\[
\det(G)=\det(G[KE(G)])\det(G[SD(G)]).
\]
\noindent We say that a graph is \emph{unimodular} if the determinant of its adjacency matrix is $\pm 1$. The above naturally reduces the problem of studying
the unimodularity of graphs to studying the unimodularity of
K\H{o}nig-Egerváry or Sterboul-Deming graphs. Unimodular Sterboul-Deming
graphs have been studied in \cite{PANE1,PANE2}. Another graph decomposition
that factorizes the determinant is the one known as the critical independence
decomposition of Larson \cite{larson2011critical}, which partitions the graph
into a K\H{o}nig-Egerváry graph and a 2-bicritical graph \cite{pulleyblank1979minimum}. 

In this work we advance the study of the unimodularity of
K\H{o}nig-Egerváry graphs via decompositions.
 
Given a matching \(M\) of \(G\), an \(M\)-\emph{perfect flower} is a pair
\((C,P)\) with the following properties. The subgraph \(C\) is an odd
cycle of length \(2k+1\) containing exactly \(k\) edges of \(M\). The path
\(P\) is a non-trivial \(M\)-alternating path, say
\(P=p_1,p_2,\ldots,p_t\), whose first and last edges belong to \(M\).
Moreover,
\[
V(C)\cap V(P)=\{p_1\}.
\]
Thus paths of length one are allowed in the definition of an
\(M\)-perfect flower; in that case the unique edge of the path is both
the first and the last edge, and hence it must belong to \(M\). Paths of
length zero are not allowed in this definition.

We define the \emph{perfect-flower part} of $G$
as the set of vertices of $G$ that belong to some $M$-perfect flower,
for some matching $M$, and we denote it by $PF(G)$, that is
\[
PF(G):=\left\{\, v\in V(G) : 
\begin{aligned}
	&v \text{ belongs to an } M\text{-perfect--flower} \\
	&\text{for some maximum matching } M
\end{aligned}
\right\}.
\]
 And define the \emph{perfect-flower-free}
\emph{part} of $G$ as $$PFF(G):=V(G)-PF(G).$$ \noindent This decomposition was first defined
in \cite{KEVINPerfectFlower} in order to study the structure of
critical independent sets in K\H{o}nig-Egerváry graphs.

In this work we prove that, for every K\H{o}nig--Egerv\'ary graph \(G\), the
perfect-flower decomposition of the vertex set is compatible with both the
determinant and the permanent of the adjacency matrix. More precisely, if
\[
V(G)=PF(G)\cup PFF(G),
\]
then
\[
\det(G)
=
\det(G[PF(G)])\det(G[PFF(G)])
\]
and
\[
\operatorname{perm}(G)
=
\operatorname{perm}(G[PF(G)])
\operatorname{perm}(G[PFF(G)]).
\]

The determinant factorization is the one relevant to unimodularity. Indeed,
within the class of K\H{o}nig--Egerv\'ary graphs, it reduces the problem of
studying whether \(\det(G)=\pm1\) to the corresponding induced subgraphs
\(G[PF(G)]\) and \(G[PFF(G)]\). In particular, this reduction separates two
induced subgraphs with different structural features: the structure of
\(G[PF(G)]\) is closely related to that of Sterboul--Deming graphs with a
perfect matching \cite{PANE1}, whereas the structure of \(G[PFF(G)]\) can
be controlled through the theory of critical independent sets
\cite{KEVINPerfectFlower}.

The permanent factorization has a complementary enumerative meaning. By
the Sachs expansion recalled in \cref{harary}, \(\operatorname{perm}(G)\) is
the positive weighted sum of the Sachs subgraphs of \(G\), where a Sachs
subgraph with \(m\) cycle components contributes \(2^m\). Hence the
permanent identity says that this weighted enumeration splits independently
over the two parts \(PF(G)\) and \(PFF(G)\). In this sense, the permanent
factorization is not merely an algebraic analogue of the determinant
factorization; it records the decomposition of the Sachs structure itself.

Consequently, the present theorem should be interpreted as an internal
reduction for the class of K\H{o}nig--Egerv\'ary graphs. For arbitrary graphs,
its use should be understood in combination with previous decompositions
that isolate a K\H{o}nig--Egerv\'ary part, such as the SD--KE decomposition
discussed above or Larson's critical independence decomposition
\cite{larson2011critical}. Thus the broader application to general graphs is
indirect: the factorization proved here applies to the K\H{o}nig--Egerv\'ary
part produced by such decompositions.

The paper is organized as follows. Section 2 contains the terminology and auxiliary results used throughout the paper. Section 3 contains the main results. Section 4 is devoted to concluding remarks and open problems.

\section{Preliminaries}\label{sss1}

All graphs considered in this paper are finite, undirected, and simple. 
For any undefined terminology or notation, we refer the reader to 
Lovász and Plummer \cite{LP} or Diestel \cite{Distel}.

Let \( G = (V, E) \) be a simple graph, where \( V = V(G) \) is the finite set of vertices and \( E = E(G) \) is the set of edges, with \( E \subseteq \{\{u, v\} : u, v \in V, u \neq v\} \). We denote the edge \( e=\{u, v\} \) as \( uv \). A subgraph of \( G \) is a graph \( H \) such that \( V(H) \subseteq V(G) \) and \( E(H) \subseteq E(G) \). A subgraph \( H \) of \( G \) is called a \textit{spanning} subgraph if \( V(H) = V(G) \). For two vertex sets $X,Y\subseteq V(G)$, we denote by $E(X,Y)$ the set of
edges $uv\in E(G)$ such that $u\in X$ and $v\in Y$.

Let \( e \in E(G) \) and \( v \in V(G) \). We define \( G - e := (V, E - \{e\}) \) and \( G - v := (V - \{v\}, \{uw \in E : u,w \neq v\}) \). If \( X \subseteq V(G) \), the \textit{induced} subgraph of \( G \) by \( X \) is the subgraph \( G[X]=(X,F) \), where \( F:=\{uv \!\in\! E(G) : u, v \!\in \! X\} \). The union of two graphs $G$ and $H$ is the graph $G\cup H$
with $V(G\cup H)=V(G)\cup V(H)$ and $E(G\cup H)=E(G)\cup E(H)$. If $M$ is a set of edges of $G$, we denote by $G[M]$ the subgraph of $G$
spanned by the edges of $M$, that is,
$V(G[M])=\{v\in V(G): \text{$v$ is an endpoint of some edge in } M\}$ and $E(G[M])=M.$

The number of vertices in a graph $G$ is called the \textit{order} of the graph and denoted by $\left|G\right|$.
A \textit{cycle} in $G$ is called \textit{odd} (resp. \textit{even}) if it has an odd (resp. even) number of edges.

A \textit{matching} \(M\) in a graph \(G\) is a set of pairwise non-adjacent edges. The \textit{matching number} of \(G\), denoted by  \(\mu(G)\), is the maximum cardinality of any matching in \(G\). Matchings induce an involution on the vertex set of the graph: \(M:V(G)\rightarrow V(G)\), where \(M(v)=u\) if \(uv \in M\), and \(M(v)=v\) otherwise. If \(S, U \subseteq V(G)\) with \(S \cap U = \emptyset\), we say that \(M\) is a matching from \(S\) to \(U\) if \(M(S) \subseteq U\). A matching $M$ is \emph{perfect} if $M(v)\neq v$ for every vertex
of the graph. Let $M$ be a matching of a graph $G$. 
A vertex set \( S \subseteq V \) is \textit{independent} if, for every pair of vertices \( u, v \in S \), we have \( uv \notin E \). 
The number of vertices in a maximum independent set is denoted by \( \alpha(G) \).  
It is well known that the symmetric difference of two matchings has a very simple structure.

\begin{lemma}\label{lematrivial}
	Let $M_{1}$ and $M_{2}$ be two maximum matchings of a graph $G$, and let
	\[
	H=(V(M_{1}\triangle M_{2}),\, M_{1}\triangle M_{2}).
	\]
	Then every connected component of $H$ is one of the following:
	\begin{enumerate}
		\item an even cycle whose edges alternate between $M_{1}$ and $M_{2}$;
		
		\item a path of even length whose edges alternate between $M_{1}$ and $M_{2}$; in this case, one end-vertex is saturated by $M_{1}$ and unsaturated by $M_{2}$, while the other is saturated by $M_{2}$ and unsaturated by $M_{1}$.
	\end{enumerate}
\end{lemma}

\section{Main Results}\label{sss2}

In this section we establish the main results of the paper. Before starting, we need to introduce some necessary tools and notation.

A spanning subgraph of a graph $G$ is called a \emph{Sachs subgraph} if
each of its components is a regular graph of degree one or two. Such
subgraphs arise naturally in the study of determinants and permanents
of adjacency matrices \cite{h2,s3}. They are also known as $\{1,2\}$-factors
and are closely related to perfect $2$-matchings. 
Note that every perfect matching is a Sachs subgraph, since all its components are copies of $K_2$. The set of all Sachs subgraphs of $G$ is denoted by $\Sa{G}$.

Let $\prk G$ be the maximum order of a subgraph $H$ of $G$ such that
$\Sa H\neq\emptyset$. We define 
\[
\SSa G:=\U_{H}\Sa H
\]
where the union is taken over all subgraphs of $G$ of order $\prk G$.
Note that when $\Sa G\neq\emptyset$ we have that 

\[
\SSa G=\Sa G.
\]
The following theorem shows that there exists a very natural relationship between the Sachs subgraphs of a graph and the computation of its determinant or permanent.

\begin{theorem}[\cite{h2,merris1981permanental}\label{harary}] 
	Let $G$ be a graph. Then
	\begin{eqnarray*}
		\det(G) & = & \sm_{H\in\Sa G}(-1)^{k(H)}2^{m(H)},\\
		\perm G & = & \sm_{H\in\Sa G}2^{m(H)}
	\end{eqnarray*}
	
	\noindent where $k(H)$ is the number of even components of $H$, and $m(H)$ is the number of cycles of $H$.
\end{theorem}

Denote by $\det(G)$ the determinant $\det(A(G))$, and define $\det(G[\emptyset])=1$.

For the reader's convenience, we include a short proof of \cref{asokpjdoas} for completeness.

\begin{lemma}[\cite{KEVINunimodular}\label{asokpjdoas}\label{asodkjasasodkjas}]
	Let $G$ be a K\H{o}nig-Egerváry graph and $H\in\SSa G,$
	then $H$ has no odd cycles. 
\end{lemma}

\begin{proof}
	Let $I$ be a maximum independent set of $G$, and set $X:=V(G)-I$. Since $G$
	is a K\H{o}nig-Egerv\'ary graph, we have
	\[
	|X|=|G|-\alpha(G)=\mu(G).
	\]
	Moreover, $X$ is a vertex cover of $G$, and therefore $X\cap V(H)$ is a
	vertex cover of $H$.
	
	Since every maximum matching of $G$ is a Sachs subgraph of order $2\mu(G)$,
	it follows that
	\[
	\prk{G}\geq 2\mu(G).
	\]
	As $H\in \SSa G$, we have $|H|=\prk{G}$, and hence
	\[
	|H|\geq 2\mu(G).
	\]
	
	Assume that $H$ has an odd cycle component. Since each component of $H$ is
	either a copy of $K_2$ or a cycle, every vertex cover of $H$ has size at least
	$|H|/2$, and in fact strictly greater than $|H|/2$ whenever $H$ has an odd
	cycle component. Therefore,
	\[
	|X\cap V(H)|>\frac{|H|}{2}.
	\]
	On the other hand,
	\[
	|X\cap V(H)|\leq |X|=\mu(G)\leq \frac{|H|}{2},
	\]
	a contradiction. Therefore $H$ has no odd cycles.
\end{proof}

\begin{corollary}\label{asodkjas}\label{asoudh1iuh23}\label{asdoasodkok2323asdoasodkok2323}
	Let $G$ be a K\H{o}nig-Egerváry graph, then
	$\prk G=2\mu(G)$.
\end{corollary}

In general, the decomposition $PFF(G),PF(G)$ does not respect the
Sachs structure of every K\H{o}nig-Egerváry graph in the following
sense: for the graph $G$ in \cref{asdioj1o2i3j} there exists $H\in\SSa G$
such that
\[
E(PFF(G),PF(G))\ii E(H)\neq\emptyset.
\]

\begin{figure}[H]
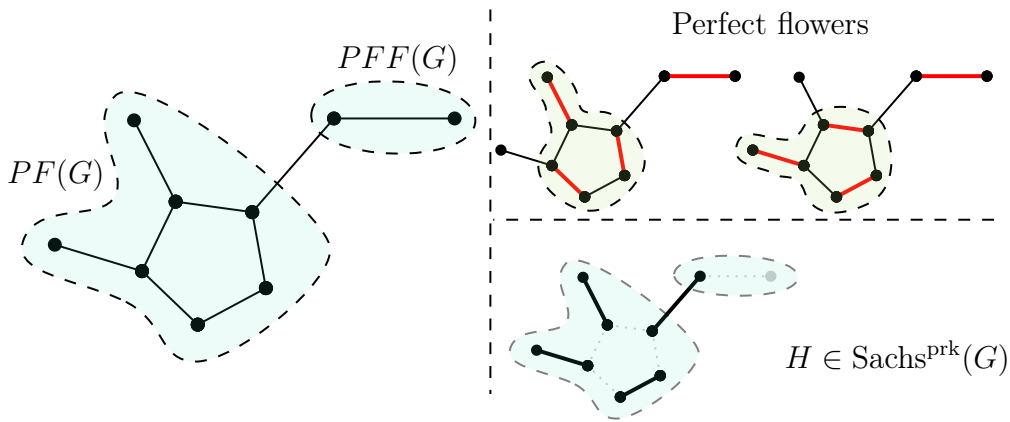

	
	\begin{center}

		\tikzset{every picture/.style={line width=0.75pt}} 
		


	\end{center}

	\caption{The partition into $PF(G)$ and $PFF(G)$, and a Sachs subgraph crossing the partition.}
	\label{asdioj1o2i3j}
	
\end{figure}

\noindent Moreover, note that $G$ satisfies:
\[
\mu(G)+\alpha(G)=4+5=9=\a G,
\]
\noindent that is, $G$ is a K\H{o}nig-Egerváry graph. However, for
K\H{o}nig-Egerváry graphs, the graph has a perfect matching if and
only if $\Sa G=\SSa G$. This reduces the study of the determinant/permanent
of the graph via \cref{harary} to the study of K\H{o}nig-Egerváry
graphs with a perfect matching. Therefore, from this point of view,
we are only interested in K\H{o}nig-Egerváry graphs with a perfect
matching. We will show in \cref{asoidj1oi23} that the decomposition
$PFF(G),PF(G)$ respects the Sachs structure in any K\H{o}nig-Egerváry
graph with a perfect matching.

\begin{theorem}\label{asoidj1oi23}\label{largo}
	Let $G$ be a K\H{o}nig-Egerváry graph with a perfect matching. Then, for
	every $H\in\Sa G$, it holds that
	\[
	E(PFF(G),PF(G))\ii E(H)=\emptyset.
	\]
\end{theorem}

\begin{proof}
	Assume for contradiction that there exists $H\in\Sa G$ such that
	\[
	E(PFF(G),PF(G))\ii E(H)\neq\emptyset.
	\]
	\noindent Let $e\in E(PFF(G),PF(G))\ii E(H)$.

	Since $H$ is a Sachs subgraph, each connected component of $H$ is either a copy of $K_2$ or a cycle. By \cref{asokpjdoas}, $H$ has no odd cycles, and therefore every cycle component of $H$ is even. Hence each cycle component has a perfect matching, and if $e$ lies on such a cycle, we may choose a perfect matching of that cycle containing $e$. If $e$ belongs to a $K_2$-component, then it is forced to belong to the unique perfect matching of that component. Choosing arbitrarily a perfect matching on every other component of $H$, we obtain a perfect matching $M_e$ of $H$ such that $e\in M_e$. Since every Sachs subgraph is spanning by definition, it follows that $M_e$ is also a perfect matching of $G$.
	By the
 choice of $e$, there exists $v\in PF(G)\ii e$. Note that $M_{e}(v)\in PFF(G)$
	and $e=vM_{e}(v)\in M_{e}$. On the other hand, there exists a perfect
	matching $M_{v}$ of $G$ such that $v$ lies in an $M_{v}$-perfect flower
	$(P,C)$ of $G$, see \cref{21312asdasd}.

	Write
	\[
	C=c_1,c_2,\ldots,c_{2r+1},c_1,
	\qquad
	P=p_1,p_2,\ldots,p_t,
	\]
	where \(p_1=c_1\). The common vertex \(p_1=c_1\) is the base of the
	\(M_v\)-blossom \(C\); otherwise \(p_1\) would be incident with an edge of
	\(M_v\) in \(C\) and also with the first edge of \(P\), which belongs to
	\(M_v\). We orient \(P\) from \(p_1\) to its other endpoint. Thus
	\[
	p_\ell p_{\ell+1}\in M_v
	\quad\Longleftrightarrow\quad
	\ell \text{ is odd},
	\qquad 1\leq \ell<t.
	\]
	In particular, since the first and last edges of \(P\) belong to \(M_v\),
	the number \(t\) is even.
	
	We also choose the cyclic orientation of \(C\) so that
	\[
	c_\ell c_{\ell+1}\in M_v
	\quad\Longleftrightarrow\quad
	\ell \text{ is even},
	\qquad 1\leq \ell\leq 2r,
	\]
	and
	\[
	c_{2r+1}c_1\notin M_v.
	\]
	Therefore \(C\) is an \(M_v\)-blossom based at \(p_1=c_1\), and
	\[
	|E(C)\cap M_v|=r=\frac{|V(C)|-1}{2}.
	\]
	Define
	\[
	x:=M_e(v).
	\]

\begin{figure}[H]
	
	\begin{center}

		\tikzset{every picture/.style={line width=0.75pt}} 
		
		\begin{tikzpicture}[x=0.75pt,y=0.75pt,yscale=-1,xscale=1]
			
			\draw    (468.63,129.34) -- (469,170.32) ;
			\draw    (468.63,129.34) -- (429.54,117.03) ;
			\draw    (405.75,150.4) -- (429.54,117.03) ;
			\draw    (405.75,150.4) -- (430.14,183.34) ;
			\draw    (469,170.32) -- (430.14,183.34) ;
			\draw    (469,170.32) ;
			\draw    (430.14,183.34) ;
			\draw    (429.54,117.03) ;
			\draw    (405.75,150.4) ;
			\draw    (405.75,150.4) ;
			\draw    (430.14,183.34) ;
			\draw    (469,170.32) ;
			\draw    (429.54,117.03) ;
			\draw    (405.75,150.4) ;
			\draw    (430.14,183.34) ;
			\draw    (429.54,117.03) ;
			\draw    (468.63,129.34) ;
			\draw    (469,170.32) ;
			\draw    (405.75,150.4) ;
			\draw [line width=0.75]    (430.14,183.34) ;
			\draw [shift={(430.14,183.34)}, rotate = 0] [color={rgb, 255:red, 0; green, 0; blue, 0 }  ][fill={rgb, 255:red, 0; green, 0; blue, 0 }  ][line width=0.75]      (0, 0) circle [x radius= 3.02, y radius= 3.02]   ;
			\draw [line width=0.75]    (405.75,150.4) ;
			\draw [shift={(405.75,150.4)}, rotate = 0] [color={rgb, 255:red, 0; green, 0; blue, 0 }  ][fill={rgb, 255:red, 0; green, 0; blue, 0 }  ][line width=0.75]      (0, 0) circle [x radius= 3.02, y radius= 3.02]   ;
			\draw [line width=0.75]    (429.54,117.03) ;
			\draw [shift={(429.54,117.03)}, rotate = 0] [color={rgb, 255:red, 0; green, 0; blue, 0 }  ][fill={rgb, 255:red, 0; green, 0; blue, 0 }  ][line width=0.75]      (0, 0) circle [x radius= 3.02, y radius= 3.02]   ;
			\draw [line width=0.75]    (468.63,129.34) ;
			\draw [shift={(468.63,129.34)}, rotate = 0] [color={rgb, 255:red, 0; green, 0; blue, 0 }  ][fill={rgb, 255:red, 0; green, 0; blue, 0 }  ][line width=0.75]      (0, 0) circle [x radius= 3.02, y radius= 3.02]   ;
			\draw [line width=0.75]    (469,170.32) ;
			\draw [shift={(469,170.32)}, rotate = 0] [color={rgb, 255:red, 0; green, 0; blue, 0 }  ][fill={rgb, 255:red, 0; green, 0; blue, 0 }  ][line width=0.75]      (0, 0) circle [x radius= 3.02, y radius= 3.02]   ;
			\draw    (357.29,150.4) -- (405.75,150.4) ;
			\draw [shift={(405.75,150.4)}, rotate = 0] [color={rgb, 255:red, 0; green, 0; blue, 0 }  ][fill={rgb, 255:red, 0; green, 0; blue, 0 }  ][line width=0.75]      (0, 0) circle [x radius= 3.02, y radius= 3.02]   ;
			\draw [shift={(357.29,150.4)}, rotate = 0] [color={rgb, 255:red, 0; green, 0; blue, 0 }  ][fill={rgb, 255:red, 0; green, 0; blue, 0 }  ][line width=0.75]      (0, 0) circle [x radius= 3.02, y radius= 3.02]   ;
			\draw    (308.83,150.4) -- (357.29,150.4) ;
			\draw [shift={(357.29,150.4)}, rotate = 0] [color={rgb, 255:red, 0; green, 0; blue, 0 }  ][fill={rgb, 255:red, 0; green, 0; blue, 0 }  ][line width=0.75]      (0, 0) circle [x radius= 3.02, y radius= 3.02]   ;
			\draw [shift={(308.83,150.4)}, rotate = 0] [color={rgb, 255:red, 0; green, 0; blue, 0 }  ][fill={rgb, 255:red, 0; green, 0; blue, 0 }  ][line width=0.75]      (0, 0) circle [x radius= 3.02, y radius= 3.02]   ;
			\draw    (260.38,150.4) -- (308.83,150.4) ;
			\draw [shift={(308.83,150.4)}, rotate = 0] [color={rgb, 255:red, 0; green, 0; blue, 0 }  ][fill={rgb, 255:red, 0; green, 0; blue, 0 }  ][line width=0.75]      (0, 0) circle [x radius= 3.02, y radius= 3.02]   ;
			\draw [shift={(260.38,150.4)}, rotate = 0] [color={rgb, 255:red, 0; green, 0; blue, 0 }  ][fill={rgb, 255:red, 0; green, 0; blue, 0 }  ][line width=0.75]      (0, 0) circle [x radius= 3.02, y radius= 3.02]   ;
			\draw    (211.92,150.4) -- (260.38,150.4) ;
			\draw [shift={(260.38,150.4)}, rotate = 0] [color={rgb, 255:red, 0; green, 0; blue, 0 }  ][fill={rgb, 255:red, 0; green, 0; blue, 0 }  ][line width=0.75]      (0, 0) circle [x radius= 3.02, y radius= 3.02]   ;
			\draw [shift={(211.92,150.4)}, rotate = 0] [color={rgb, 255:red, 0; green, 0; blue, 0 }  ][fill={rgb, 255:red, 0; green, 0; blue, 0 }  ][line width=0.75]      (0, 0) circle [x radius= 3.02, y radius= 3.02]   ;
			\draw    (163.46,150.4) -- (183.81,150.4) -- (211.92,150.4) ;
			\draw [shift={(211.92,150.4)}, rotate = 0] [color={rgb, 255:red, 0; green, 0; blue, 0 }  ][fill={rgb, 255:red, 0; green, 0; blue, 0 }  ][line width=0.75]      (0, 0) circle [x radius= 3.02, y radius= 3.02]   ;
			\draw [shift={(163.46,150.4)}, rotate = 0] [color={rgb, 255:red, 0; green, 0; blue, 0 }  ][fill={rgb, 255:red, 0; green, 0; blue, 0 }  ][line width=0.75]      (0, 0) circle [x radius= 3.02, y radius= 3.02]   ;
			\draw    (260.38,150.4) -- (260.38,75) ;
			\draw [shift={(260.38,75)}, rotate = 270] [color={rgb, 255:red, 0; green, 0; blue, 0 }  ][fill={rgb, 255:red, 0; green, 0; blue, 0 }  ][line width=0.75]      (0, 0) circle [x radius= 3.02, y radius= 3.02]   ;
			\draw [shift={(260.38,150.4)}, rotate = 270] [color={rgb, 255:red, 0; green, 0; blue, 0 }  ][fill={rgb, 255:red, 0; green, 0; blue, 0 }  ][line width=0.75]      (0, 0) circle [x radius= 3.02, y radius= 3.02]   ;
			\draw [color={rgb, 255:red, 255; green, 0; blue, 0 }  ,draw opacity=1 ][line width=1.5]    (163.46,150.4) -- (211.92,150.4) ;
			\draw [color={rgb, 255:red, 255; green, 0; blue, 0 }  ,draw opacity=1 ][line width=1.5]    (260.38,150.4) -- (308.83,150.4) ;
			\draw [color={rgb, 255:red, 255; green, 0; blue, 0 }  ,draw opacity=1 ][line width=1.5]    (357.29,150.4) -- (405.75,150.4) ;
			\draw [color={rgb, 255:red, 255; green, 0; blue, 0 }  ,draw opacity=1 ][line width=1.5]    (429.54,117.03) -- (468.63,129.34) ;
			\draw [color={rgb, 255:red, 255; green, 0; blue, 0 }  ,draw opacity=1 ][line width=1.5]    (430.14,183.34) -- (469,170.32) ;
			\draw [color={rgb, 255:red, 0; green, 0; blue, 255 }  ,draw opacity=1 ][line width=1.5]    (260.38,150.4) -- (260.38,75) ;
			\draw    (163.46,150.4) ;
			\draw [shift={(163.46,150.4)}, rotate = 0] [color={rgb, 255:red, 0; green, 0; blue, 0 }  ][fill={rgb, 255:red, 0; green, 0; blue, 0 }  ][line width=0.75]      (0, 0) circle [x radius= 3.02, y radius= 3.02]   ;
			\draw [shift={(163.46,150.4)}, rotate = 0] [color={rgb, 255:red, 0; green, 0; blue, 0 }  ][fill={rgb, 255:red, 0; green, 0; blue, 0 }  ][line width=0.75]      (0, 0) circle [x radius= 3.02, y radius= 3.02]   ;
			\draw    (211.92,150.4) ;
			\draw [shift={(211.92,150.4)}, rotate = 0] [color={rgb, 255:red, 0; green, 0; blue, 0 }  ][fill={rgb, 255:red, 0; green, 0; blue, 0 }  ][line width=0.75]      (0, 0) circle [x radius= 3.02, y radius= 3.02]   ;
			\draw [shift={(211.92,150.4)}, rotate = 0] [color={rgb, 255:red, 0; green, 0; blue, 0 }  ][fill={rgb, 255:red, 0; green, 0; blue, 0 }  ][line width=0.75]      (0, 0) circle [x radius= 3.02, y radius= 3.02]   ;
			\draw    (430.14,183.34) ;
			\draw [shift={(430.14,183.34)}, rotate = 0] [color={rgb, 255:red, 0; green, 0; blue, 0 }  ][fill={rgb, 255:red, 0; green, 0; blue, 0 }  ][line width=0.75]      (0, 0) circle [x radius= 3.02, y radius= 3.02]   ;
			\draw [shift={(430.14,183.34)}, rotate = 0] [color={rgb, 255:red, 0; green, 0; blue, 0 }  ][fill={rgb, 255:red, 0; green, 0; blue, 0 }  ][line width=0.75]      (0, 0) circle [x radius= 3.02, y radius= 3.02]   ;
			\draw    (469,170.32) ;
			\draw [shift={(469,170.32)}, rotate = 0] [color={rgb, 255:red, 0; green, 0; blue, 0 }  ][fill={rgb, 255:red, 0; green, 0; blue, 0 }  ][line width=0.75]      (0, 0) circle [x radius= 3.02, y radius= 3.02]   ;
			\draw [shift={(469,170.32)}, rotate = 0] [color={rgb, 255:red, 0; green, 0; blue, 0 }  ][fill={rgb, 255:red, 0; green, 0; blue, 0 }  ][line width=0.75]      (0, 0) circle [x radius= 3.02, y radius= 3.02]   ;
			\draw    (468.63,129.34) ;
			\draw [shift={(468.63,129.34)}, rotate = 0] [color={rgb, 255:red, 0; green, 0; blue, 0 }  ][fill={rgb, 255:red, 0; green, 0; blue, 0 }  ][line width=0.75]      (0, 0) circle [x radius= 3.02, y radius= 3.02]   ;
			\draw [shift={(468.63,129.34)}, rotate = 0] [color={rgb, 255:red, 0; green, 0; blue, 0 }  ][fill={rgb, 255:red, 0; green, 0; blue, 0 }  ][line width=0.75]      (0, 0) circle [x radius= 3.02, y radius= 3.02]   ;
			\draw    (429.54,117.03) ;
			\draw [shift={(429.54,117.03)}, rotate = 0] [color={rgb, 255:red, 0; green, 0; blue, 0 }  ][fill={rgb, 255:red, 0; green, 0; blue, 0 }  ][line width=0.75]      (0, 0) circle [x radius= 3.02, y radius= 3.02]   ;
			\draw [shift={(429.54,117.03)}, rotate = 0] [color={rgb, 255:red, 0; green, 0; blue, 0 }  ][fill={rgb, 255:red, 0; green, 0; blue, 0 }  ][line width=0.75]      (0, 0) circle [x radius= 3.02, y radius= 3.02]   ;
			\draw    (405.75,150.4) ;
			\draw [shift={(405.75,150.4)}, rotate = 0] [color={rgb, 255:red, 0; green, 0; blue, 0 }  ][fill={rgb, 255:red, 0; green, 0; blue, 0 }  ][line width=0.75]      (0, 0) circle [x radius= 3.02, y radius= 3.02]   ;
			\draw [shift={(405.75,150.4)}, rotate = 0] [color={rgb, 255:red, 0; green, 0; blue, 0 }  ][fill={rgb, 255:red, 0; green, 0; blue, 0 }  ][line width=0.75]      (0, 0) circle [x radius= 3.02, y radius= 3.02]   ;
			\draw    (357.29,150.4) ;
			\draw [shift={(357.29,150.4)}, rotate = 0] [color={rgb, 255:red, 0; green, 0; blue, 0 }  ][fill={rgb, 255:red, 0; green, 0; blue, 0 }  ][line width=0.75]      (0, 0) circle [x radius= 3.02, y radius= 3.02]   ;
			\draw [shift={(357.29,150.4)}, rotate = 0] [color={rgb, 255:red, 0; green, 0; blue, 0 }  ][fill={rgb, 255:red, 0; green, 0; blue, 0 }  ][line width=0.75]      (0, 0) circle [x radius= 3.02, y radius= 3.02]   ;
			\draw    (308.83,150.4) ;
			\draw [shift={(308.83,150.4)}, rotate = 0] [color={rgb, 255:red, 0; green, 0; blue, 0 }  ][fill={rgb, 255:red, 0; green, 0; blue, 0 }  ][line width=0.75]      (0, 0) circle [x radius= 3.02, y radius= 3.02]   ;
			\draw [shift={(308.83,150.4)}, rotate = 0] [color={rgb, 255:red, 0; green, 0; blue, 0 }  ][fill={rgb, 255:red, 0; green, 0; blue, 0 }  ][line width=0.75]      (0, 0) circle [x radius= 3.02, y radius= 3.02]   ;
			\draw    (260.38,75) ;
			\draw [shift={(260.38,75)}, rotate = 0] [color={rgb, 255:red, 0; green, 0; blue, 0 }  ][fill={rgb, 255:red, 0; green, 0; blue, 0 }  ][line width=0.75]      (0, 0) circle [x radius= 3.02, y radius= 3.02]   ;
			\draw [shift={(260.38,75)}, rotate = 0] [color={rgb, 255:red, 0; green, 0; blue, 0 }  ][fill={rgb, 255:red, 0; green, 0; blue, 0 }  ][line width=0.75]      (0, 0) circle [x radius= 3.02, y radius= 3.02]   ;
			\draw    (260.38,150.4) ;
			\draw [shift={(260.38,150.4)}, rotate = 0] [color={rgb, 255:red, 0; green, 0; blue, 0 }  ][fill={rgb, 255:red, 0; green, 0; blue, 0 }  ][line width=0.75]      (0, 0) circle [x radius= 3.02, y radius= 3.02]   ;
			\draw [shift={(260.38,150.4)}, rotate = 0] [color={rgb, 255:red, 0; green, 0; blue, 0 }  ][fill={rgb, 255:red, 0; green, 0; blue, 0 }  ][line width=0.75]      (0, 0) circle [x radius= 3.02, y radius= 3.02]   ;
			\draw   (160,188) .. controls (160,192.67) and (162.33,195) .. (167,195) -- (274.5,195) .. controls (281.17,195) and (284.5,197.33) .. (284.5,202) .. controls (284.5,197.33) and (287.83,195) .. (294.5,195)(291.5,195) -- (402,195) .. controls (406.67,195) and (409,192.67) .. (409,188) ;
			\draw   (486,96) .. controls (486,91.33) and (483.67,89) .. (479,89) -- (448.5,89) .. controls (441.83,89) and (438.5,86.67) .. (438.5,82) .. controls (438.5,86.67) and (435.17,89) .. (428.5,89)(431.5,89) -- (398,89) .. controls (393.33,89) and (391,91.33) .. (391,96) ;
			\draw   (154,110) .. controls (149.33,110) and (147,112.33) .. (147,117) -- (147,151.21) .. controls (147,157.88) and (144.67,161.21) .. (140,161.21) .. controls (144.67,161.21) and (147,164.54) .. (147,171.21)(147,168.21) -- (147,205.43) .. controls (147,210.1) and (149.33,212.43) .. (154,212.43) ;
			\draw   (154,52.43) .. controls (149.33,52.43) and (147,54.76) .. (147,59.43) -- (147,69.93) .. controls (147,76.6) and (144.67,79.93) .. (140,79.93) .. controls (144.67,79.93) and (147,83.26) .. (147,89.93)(147,86.93) -- (147,100.43) .. controls (147,105.1) and (149.33,107.43) .. (154,107.43) ;
			
			\draw (397,156.4) node [anchor=north west][inner sep=0.75pt]  [font=\small]  {$p_{1}$};
			\draw (349,156.4) node [anchor=north west][inner sep=0.75pt]  [font=\small]  {$p_{2}$};
			\draw (301,156.4) node [anchor=north west][inner sep=0.75pt]  [font=\small]  {$p_{3}$};
			\draw (253,156.4) node [anchor=north west][inner sep=0.75pt]  [font=\small]  {$p_{4}$};
			\draw (204,156.4) node [anchor=north west][inner sep=0.75pt]  [font=\small]  {$p_{5}$};
			\draw (156,157.4) node [anchor=north west][inner sep=0.75pt]  [font=\small]  {$p_{6}$};
			\draw (244,69.4) node [anchor=north west][inner sep=0.75pt]  [font=\small]  {$x$};
			\draw (267,67.4) node [anchor=north west][inner sep=0.75pt]  [font=\footnotesize]  {$M_{e}( v)$};
			\draw (292,129.4) node [anchor=north west][inner sep=0.75pt]  [font=\footnotesize]  {$M_{v}( v)$};
			\draw (248,135.4) node [anchor=north west][inner sep=0.75pt]  [font=\small]  {$v$};
			\draw (397,131.4) node [anchor=north west][inner sep=0.75pt]  [font=\small]  {$c_{1}$};
			\draw (421,99.4) node [anchor=north west][inner sep=0.75pt]  [font=\small]  {$c_{2}$};
			\draw (472,119.4) node [anchor=north west][inner sep=0.75pt]  [font=\small]  {$c_{3}$};
			\draw (473,165.4) node [anchor=north west][inner sep=0.75pt]  [font=\small]  {$c_{4}$};
			\draw (422,188.4) node [anchor=north west][inner sep=0.75pt]  [font=\small]  {$c_{5}$};
			\draw (277,207.4) node [anchor=north west][inner sep=0.75pt]    {$P$};
			\draw (432,62.4) node [anchor=north west][inner sep=0.75pt]    {$C$};
			\draw (87,153.4) node [anchor=north west][inner sep=0.75pt]    {$PF( G)$};
			\draw (75,71.4) node [anchor=north west][inner sep=0.75pt]    {$PFF( G)$};
			\draw (264,106.4) node [anchor=north west][inner sep=0.75pt]  [font=\normalsize]  {$e$};

		\end{tikzpicture}

	\end{center}

\caption{An edge between $PF(G)$ and $PFF(G)$, and an $M_v$-perfect flower containing $v$.}
	\label{21312asdasd}
	
\end{figure}
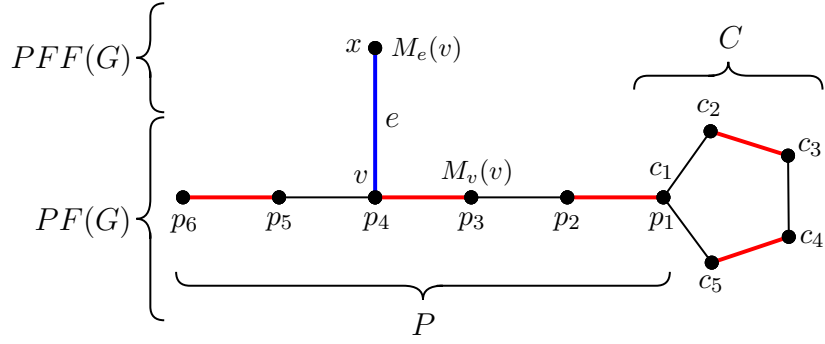

\noindent Then note that the edges $vM_v(v)$ and $vx$ are distinct. Set
\[
S:=V(P)\cup V(C).
\]
Since $(P,C)$ is an $M_v$-perfect flower, every vertex of $S$ is matched by $M_v$ to another vertex of $S$. In particular, no edge of $M_v$ has exactly one endpoint in $S$. By \cref{lematrivial}, the connected component $K$ of
\[
\bigl(V(M_v\triangle M_e),\, M_v\triangle M_e\bigr)
\]
containing $v$ is an even cycle whose edges alternate between $M_v$ and $M_e$. Since $x=M_e(v)\notin S$, if we traverse $K$ starting at $v$ along the edge $vx\in M_e$, we eventually meet a vertex of $S$ again, because $K$ is a cycle and $v\in S$. Let $u$ be the first such vertex after $v$, and let
\[
Q=q_1,\dots,q_w
\]
be the corresponding subpath of $K$, where $q_1=v$ and $q_w=u$. By construction,
\[
V(Q)\cap S=\{q_1,q_w\}=\{v,u\}.
\]
Moreover, \(Q\) alternates between edges of \(M_e\) and edges of \(M_v\),
and its first edge is
\[
q_1q_2=vx\in M_e.
\]
Since the internal vertices of \(Q\) lie outside \(S\) and no edge of
\(M_v\) joins \(S\) to \(V(G)\setminus S\), the last edge \(q_{w-1}q_w\)
cannot belong to \(M_v\). Hence
\[
q_{w-1}q_w\in M_e.
\]
Therefore,
\[
q_\ell q_{\ell+1}\in M_e
\quad\Longleftrightarrow\quad
\ell \text{ is odd},
\]
and
\[
q_\ell q_{\ell+1}\in M_v
\quad\Longleftrightarrow\quad
\ell \text{ is even},
\qquad 1\leq \ell<w.
\]
In particular, \(w\) is even. Since \(q_2=x\notin V(P)\cup V(C)\) and
\(q_w=u\in V(P)\cup V(C)\), we have \(w\geq 4\). Hence \(q_3\) is defined
and
\[
q_3\notin V(P)\cup V(C).
\]

We now consider the possible positions of \(v\) and \(u\) in
\(V(P)\cup V(C)\). We keep the roles of \(v\) and \(u\) fixed throughout the
argument: the vertex \(v\) is the endpoint of \(e\) that belongs to \(PF(G)\),
whereas \(x=q_2=M_e(v)\) is the vertex that must be forced to belong to
\(PF(G)\). Thus no case below is obtained by interchanging \(u\) and \(v\).

In each case, we explicitly identify an odd cycle \(D\) which is an
\(M\)-blossom for the relevant perfect matching \(M\), and an
\(M\)-alternating path \(R\) whose first and last edges belong to \(M\). We
also check that \(V(D)\cap V(R)\) consists of exactly one vertex.

$ $

\textbf{Case 1.} Assume that $v\in V(P)$, say $v=p_{i}$
for some odd $i$. According to the position of the vertex $u$, we
consider the following subcases.

$ $

\textbf{Case 1.1.} Assume that \(u\in V(P)\), say \(u=p_j\) for some even \(j\),
see \cref{12iuh3u1i2h}. Consider the path
\[
R:=p_1,p_2,\ldots,p_j(=q_w),q_{w-1},q_{w-2},\ldots,q_2 .
\]
The cycle \(C\) is an \(M_v\)-blossom based at \(p_1=c_1\). Let us check
that \(R\) is a valid attaching path. Its first edge is
\(p_1p_2\in M_v\). Since \(j\) is even, the edge \(p_{j-1}p_j\) also belongs
to \(M_v\). The next edge, namely \(q_wq_{w-1}\), does not belong to
\(M_v\), because \(q_{w-1}q_w\in M_e\). Along the reversed part of \(Q\),
the edges alternate with respect to \(M_v\), and the last edge
\(q_3q_2\) belongs to \(M_v\). Hence \(R\) is \(M_v\)-alternating and its
first and last edges belong to \(M_v\).

Moreover,
\[
V(C)\cap V(R)=\{p_1\},
\]
because \(V(C)\cap V(P)=\{p_1\}\) and the vertices
\(q_2,q_3,\ldots,q_{w-1}\) lie outside \(V(P)\cup V(C)\). Therefore
\(C\) together with \(R\) is an \(M_v\)-perfect flower containing
\(x=q_2\). This contradicts \(x\in PFF(G)\).

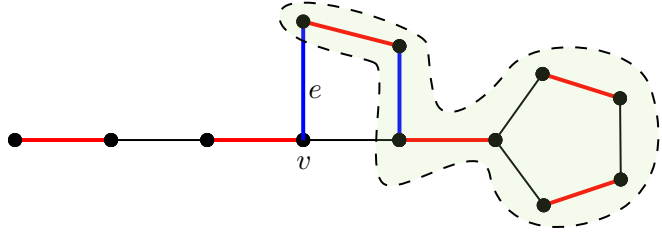
\begin{figure}[H]
	
	\begin{center}

		\tikzset{every picture/.style={line width=0.75pt}} 
		
		\begin{tikzpicture}[x=0.75pt,y=0.75pt,yscale=-1,xscale=1]
			
			\draw    (458.63,155.34) -- (459,196.32) ;
			\draw    (458.63,155.34) -- (419.54,143.03) ;
			\draw    (395.75,176.4) -- (419.54,143.03) ;
			\draw    (395.75,176.4) -- (420.14,209.34) ;
			\draw    (459,196.32) -- (420.14,209.34) ;
			\draw    (459,196.32) ;
			\draw    (420.14,209.34) ;
			\draw    (419.54,143.03) ;
			\draw    (395.75,176.4) ;
			\draw    (395.75,176.4) ;
			\draw    (420.14,209.34) ;
			\draw    (459,196.32) ;
			\draw    (419.54,143.03) ;
			\draw    (395.75,176.4) ;
			\draw    (420.14,209.34) ;
			\draw    (419.54,143.03) ;
			\draw    (458.63,155.34) ;
			\draw    (459,196.32) ;
			\draw    (395.75,176.4) ;
			\draw [line width=0.75]    (420.14,209.34) ;
			\draw [shift={(420.14,209.34)}, rotate = 0] [color={rgb, 255:red, 0; green, 0; blue, 0 }  ][fill={rgb, 255:red, 0; green, 0; blue, 0 }  ][line width=0.75]      (0, 0) circle [x radius= 3.02, y radius= 3.02]   ;
			\draw [line width=0.75]    (395.75,176.4) ;
			\draw [shift={(395.75,176.4)}, rotate = 0] [color={rgb, 255:red, 0; green, 0; blue, 0 }  ][fill={rgb, 255:red, 0; green, 0; blue, 0 }  ][line width=0.75]      (0, 0) circle [x radius= 3.02, y radius= 3.02]   ;
			\draw [line width=0.75]    (419.54,143.03) ;
			\draw [shift={(419.54,143.03)}, rotate = 0] [color={rgb, 255:red, 0; green, 0; blue, 0 }  ][fill={rgb, 255:red, 0; green, 0; blue, 0 }  ][line width=0.75]      (0, 0) circle [x radius= 3.02, y radius= 3.02]   ;
			\draw [line width=0.75]    (458.63,155.34) ;
			\draw [shift={(458.63,155.34)}, rotate = 0] [color={rgb, 255:red, 0; green, 0; blue, 0 }  ][fill={rgb, 255:red, 0; green, 0; blue, 0 }  ][line width=0.75]      (0, 0) circle [x radius= 3.02, y radius= 3.02]   ;
			\draw [line width=0.75]    (459,196.32) ;
			\draw [shift={(459,196.32)}, rotate = 0] [color={rgb, 255:red, 0; green, 0; blue, 0 }  ][fill={rgb, 255:red, 0; green, 0; blue, 0 }  ][line width=0.75]      (0, 0) circle [x radius= 3.02, y radius= 3.02]   ;
			\draw    (347.29,176.4) -- (395.75,176.4) ;
			\draw [shift={(395.75,176.4)}, rotate = 0] [color={rgb, 255:red, 0; green, 0; blue, 0 }  ][fill={rgb, 255:red, 0; green, 0; blue, 0 }  ][line width=0.75]      (0, 0) circle [x radius= 3.02, y radius= 3.02]   ;
			\draw [shift={(347.29,176.4)}, rotate = 0] [color={rgb, 255:red, 0; green, 0; blue, 0 }  ][fill={rgb, 255:red, 0; green, 0; blue, 0 }  ][line width=0.75]      (0, 0) circle [x radius= 3.02, y radius= 3.02]   ;
			\draw    (298.83,176.4) -- (347.29,176.4) ;
			\draw [shift={(347.29,176.4)}, rotate = 0] [color={rgb, 255:red, 0; green, 0; blue, 0 }  ][fill={rgb, 255:red, 0; green, 0; blue, 0 }  ][line width=0.75]      (0, 0) circle [x radius= 3.02, y radius= 3.02]   ;
			\draw [shift={(298.83,176.4)}, rotate = 0] [color={rgb, 255:red, 0; green, 0; blue, 0 }  ][fill={rgb, 255:red, 0; green, 0; blue, 0 }  ][line width=0.75]      (0, 0) circle [x radius= 3.02, y radius= 3.02]   ;
			\draw    (250.38,176.4) -- (298.83,176.4) ;
			\draw [shift={(298.83,176.4)}, rotate = 0] [color={rgb, 255:red, 0; green, 0; blue, 0 }  ][fill={rgb, 255:red, 0; green, 0; blue, 0 }  ][line width=0.75]      (0, 0) circle [x radius= 3.02, y radius= 3.02]   ;
			\draw [shift={(250.38,176.4)}, rotate = 0] [color={rgb, 255:red, 0; green, 0; blue, 0 }  ][fill={rgb, 255:red, 0; green, 0; blue, 0 }  ][line width=0.75]      (0, 0) circle [x radius= 3.02, y radius= 3.02]   ;
			\draw    (201.92,176.4) -- (250.38,176.4) ;
			\draw [shift={(250.38,176.4)}, rotate = 0] [color={rgb, 255:red, 0; green, 0; blue, 0 }  ][fill={rgb, 255:red, 0; green, 0; blue, 0 }  ][line width=0.75]      (0, 0) circle [x radius= 3.02, y radius= 3.02]   ;
			\draw [shift={(201.92,176.4)}, rotate = 0] [color={rgb, 255:red, 0; green, 0; blue, 0 }  ][fill={rgb, 255:red, 0; green, 0; blue, 0 }  ][line width=0.75]      (0, 0) circle [x radius= 3.02, y radius= 3.02]   ;
			\draw    (153.46,176.4) -- (173.81,176.4) -- (201.92,176.4) ;
			\draw [shift={(201.92,176.4)}, rotate = 0] [color={rgb, 255:red, 0; green, 0; blue, 0 }  ][fill={rgb, 255:red, 0; green, 0; blue, 0 }  ][line width=0.75]      (0, 0) circle [x radius= 3.02, y radius= 3.02]   ;
			\draw [shift={(153.46,176.4)}, rotate = 0] [color={rgb, 255:red, 0; green, 0; blue, 0 }  ][fill={rgb, 255:red, 0; green, 0; blue, 0 }  ][line width=0.75]      (0, 0) circle [x radius= 3.02, y radius= 3.02]   ;
			\draw    (347.38,176.4) -- (347.38,129) ;
			\draw [shift={(347.38,129)}, rotate = 270] [color={rgb, 255:red, 0; green, 0; blue, 0 }  ][fill={rgb, 255:red, 0; green, 0; blue, 0 }  ][line width=0.75]      (0, 0) circle [x radius= 3.02, y radius= 3.02]   ;
			\draw [shift={(347.38,176.4)}, rotate = 270] [color={rgb, 255:red, 0; green, 0; blue, 0 }  ][fill={rgb, 255:red, 0; green, 0; blue, 0 }  ][line width=0.75]      (0, 0) circle [x radius= 3.02, y radius= 3.02]   ;
			\draw [color={rgb, 255:red, 255; green, 0; blue, 0 }  ,draw opacity=1 ][line width=1.5]    (153.46,176.4) -- (201.92,176.4) ;
			\draw [color={rgb, 255:red, 255; green, 0; blue, 0 }  ,draw opacity=1 ][line width=1.5]    (250.38,176.4) -- (298.83,176.4) ;
			\draw [color={rgb, 255:red, 255; green, 0; blue, 0 }  ,draw opacity=1 ][line width=1.5]    (347.29,176.4) -- (395.75,176.4) ;
			\draw [color={rgb, 255:red, 255; green, 0; blue, 0 }  ,draw opacity=1 ][line width=1.5]    (419.54,143.03) -- (458.63,155.34) ;
			\draw [color={rgb, 255:red, 255; green, 0; blue, 0 }  ,draw opacity=1 ][line width=1.5]    (420.14,209.34) -- (459,196.32) ;
			\draw [color={rgb, 255:red, 0; green, 0; blue, 255 }  ,draw opacity=1 ][line width=1.5]    (347.38,176.4) -- (347.38,129) ;
			\draw    (153.46,176.4) ;
			\draw [shift={(153.46,176.4)}, rotate = 0] [color={rgb, 255:red, 0; green, 0; blue, 0 }  ][fill={rgb, 255:red, 0; green, 0; blue, 0 }  ][line width=0.75]      (0, 0) circle [x radius= 3.02, y radius= 3.02]   ;
			\draw [shift={(153.46,176.4)}, rotate = 0] [color={rgb, 255:red, 0; green, 0; blue, 0 }  ][fill={rgb, 255:red, 0; green, 0; blue, 0 }  ][line width=0.75]      (0, 0) circle [x radius= 3.02, y radius= 3.02]   ;
			\draw    (201.92,176.4) ;
			\draw [shift={(201.92,176.4)}, rotate = 0] [color={rgb, 255:red, 0; green, 0; blue, 0 }  ][fill={rgb, 255:red, 0; green, 0; blue, 0 }  ][line width=0.75]      (0, 0) circle [x radius= 3.02, y radius= 3.02]   ;
			\draw [shift={(201.92,176.4)}, rotate = 0] [color={rgb, 255:red, 0; green, 0; blue, 0 }  ][fill={rgb, 255:red, 0; green, 0; blue, 0 }  ][line width=0.75]      (0, 0) circle [x radius= 3.02, y radius= 3.02]   ;
			\draw    (420.14,209.34) ;
			\draw [shift={(420.14,209.34)}, rotate = 0] [color={rgb, 255:red, 0; green, 0; blue, 0 }  ][fill={rgb, 255:red, 0; green, 0; blue, 0 }  ][line width=0.75]      (0, 0) circle [x radius= 3.02, y radius= 3.02]   ;
			\draw [shift={(420.14,209.34)}, rotate = 0] [color={rgb, 255:red, 0; green, 0; blue, 0 }  ][fill={rgb, 255:red, 0; green, 0; blue, 0 }  ][line width=0.75]      (0, 0) circle [x radius= 3.02, y radius= 3.02]   ;
			\draw    (459,196.32) ;
			\draw [shift={(459,196.32)}, rotate = 0] [color={rgb, 255:red, 0; green, 0; blue, 0 }  ][fill={rgb, 255:red, 0; green, 0; blue, 0 }  ][line width=0.75]      (0, 0) circle [x radius= 3.02, y radius= 3.02]   ;
			\draw [shift={(459,196.32)}, rotate = 0] [color={rgb, 255:red, 0; green, 0; blue, 0 }  ][fill={rgb, 255:red, 0; green, 0; blue, 0 }  ][line width=0.75]      (0, 0) circle [x radius= 3.02, y radius= 3.02]   ;
			\draw    (458.63,155.34) ;
			\draw [shift={(458.63,155.34)}, rotate = 0] [color={rgb, 255:red, 0; green, 0; blue, 0 }  ][fill={rgb, 255:red, 0; green, 0; blue, 0 }  ][line width=0.75]      (0, 0) circle [x radius= 3.02, y radius= 3.02]   ;
			\draw [shift={(458.63,155.34)}, rotate = 0] [color={rgb, 255:red, 0; green, 0; blue, 0 }  ][fill={rgb, 255:red, 0; green, 0; blue, 0 }  ][line width=0.75]      (0, 0) circle [x radius= 3.02, y radius= 3.02]   ;
			\draw    (419.54,143.03) ;
			\draw [shift={(419.54,143.03)}, rotate = 0] [color={rgb, 255:red, 0; green, 0; blue, 0 }  ][fill={rgb, 255:red, 0; green, 0; blue, 0 }  ][line width=0.75]      (0, 0) circle [x radius= 3.02, y radius= 3.02]   ;
			\draw [shift={(419.54,143.03)}, rotate = 0] [color={rgb, 255:red, 0; green, 0; blue, 0 }  ][fill={rgb, 255:red, 0; green, 0; blue, 0 }  ][line width=0.75]      (0, 0) circle [x radius= 3.02, y radius= 3.02]   ;
			\draw    (395.75,176.4) ;
			\draw [shift={(395.75,176.4)}, rotate = 0] [color={rgb, 255:red, 0; green, 0; blue, 0 }  ][fill={rgb, 255:red, 0; green, 0; blue, 0 }  ][line width=0.75]      (0, 0) circle [x radius= 3.02, y radius= 3.02]   ;
			\draw [shift={(395.75,176.4)}, rotate = 0] [color={rgb, 255:red, 0; green, 0; blue, 0 }  ][fill={rgb, 255:red, 0; green, 0; blue, 0 }  ][line width=0.75]      (0, 0) circle [x radius= 3.02, y radius= 3.02]   ;
			\draw    (347.29,176.4) ;
			\draw [shift={(347.29,176.4)}, rotate = 0] [color={rgb, 255:red, 0; green, 0; blue, 0 }  ][fill={rgb, 255:red, 0; green, 0; blue, 0 }  ][line width=0.75]      (0, 0) circle [x radius= 3.02, y radius= 3.02]   ;
			\draw [shift={(347.29,176.4)}, rotate = 0] [color={rgb, 255:red, 0; green, 0; blue, 0 }  ][fill={rgb, 255:red, 0; green, 0; blue, 0 }  ][line width=0.75]      (0, 0) circle [x radius= 3.02, y radius= 3.02]   ;
			\draw    (298.83,176.4) ;
			\draw [shift={(298.83,176.4)}, rotate = 0] [color={rgb, 255:red, 0; green, 0; blue, 0 }  ][fill={rgb, 255:red, 0; green, 0; blue, 0 }  ][line width=0.75]      (0, 0) circle [x radius= 3.02, y radius= 3.02]   ;
			\draw [shift={(298.83,176.4)}, rotate = 0] [color={rgb, 255:red, 0; green, 0; blue, 0 }  ][fill={rgb, 255:red, 0; green, 0; blue, 0 }  ][line width=0.75]      (0, 0) circle [x radius= 3.02, y radius= 3.02]   ;
			\draw    (250.38,176.4) ;
			\draw [shift={(250.38,176.4)}, rotate = 0] [color={rgb, 255:red, 0; green, 0; blue, 0 }  ][fill={rgb, 255:red, 0; green, 0; blue, 0 }  ][line width=0.75]      (0, 0) circle [x radius= 3.02, y radius= 3.02]   ;
			\draw [shift={(250.38,176.4)}, rotate = 0] [color={rgb, 255:red, 0; green, 0; blue, 0 }  ][fill={rgb, 255:red, 0; green, 0; blue, 0 }  ][line width=0.75]      (0, 0) circle [x radius= 3.02, y radius= 3.02]   ;
			\draw [color={rgb, 255:red, 255; green, 0; blue, 0 }  ,draw opacity=1 ][line width=1.5]    (298.83,116.57) -- (347.38,129) ;
			\draw [color={rgb, 255:red, 0; green, 0; blue, 255 }  ,draw opacity=1 ][line width=1.5]    (298.83,176.4) -- (298.83,116.57) ;
			\draw    (298.83,116.57) ;
			\draw [shift={(298.83,116.57)}, rotate = 0] [color={rgb, 255:red, 0; green, 0; blue, 0 }  ][fill={rgb, 255:red, 0; green, 0; blue, 0 }  ][line width=0.75]      (0, 0) circle [x radius= 3.02, y radius= 3.02]   ;
			\draw [shift={(298.83,116.57)}, rotate = 0] [color={rgb, 255:red, 0; green, 0; blue, 0 }  ][fill={rgb, 255:red, 0; green, 0; blue, 0 }  ][line width=0.75]      (0, 0) circle [x radius= 3.02, y radius= 3.02]   ;
			\draw    (347.38,129) ;
			\draw [shift={(347.38,129)}, rotate = 0] [color={rgb, 255:red, 0; green, 0; blue, 0 }  ][fill={rgb, 255:red, 0; green, 0; blue, 0 }  ][line width=0.75]      (0, 0) circle [x radius= 3.02, y radius= 3.02]   ;
			\draw [shift={(347.38,129)}, rotate = 0] [color={rgb, 255:red, 0; green, 0; blue, 0 }  ][fill={rgb, 255:red, 0; green, 0; blue, 0 }  ][line width=0.75]      (0, 0) circle [x radius= 3.02, y radius= 3.02]   ;
			\draw  [fill={rgb, 255:red, 184; green, 233; blue, 134 }  ,fill opacity=0.15 ][dash pattern={on 4.5pt off 4.5pt}] (335,137.57) .. controls (328,133.57) and (284,128.57) .. (287,112.57) .. controls (290,96.57) and (359,119.57) .. (361,127.57) .. controls (363,135.57) and (362,146.57) .. (366,157.57) .. controls (370,168.57) and (389,159.14) .. (402,141.14) .. controls (415,123.14) and (466,126.14) .. (474,150.14) .. controls (482,174.14) and (480,205.14) .. (449,216.14) .. controls (418,227.14) and (397,215.14) .. (393,194.14) .. controls (389,173.14) and (348,206.14) .. (339,198.14) .. controls (330,190.14) and (342,141.57) .. (335,137.57) -- cycle ;
			
			\draw (294,183.4) node [anchor=north west][inner sep=0.75pt]  [font=\small]  {$v$};
			\draw (300,147.4) node [anchor=north west][inner sep=0.75pt]  [font=\small]  {$e$};

		\end{tikzpicture}

	\end{center}

\caption{Case 1.1: both vertices lie on $P$.}
	\label{12iuh3u1i2h}
	
\end{figure}

\textbf{Case 1.2.} Assume that \(u\in V(P)\), say \(u=p_j\) for some odd \(j\),
see \cref{j123oi12j3} for the case \(j>i\). Since \(u\neq v\), we have \(i\neq j\). Let
\[
a:=\min\{i,j\},\qquad b:=\max\{i,j\},\qquad z:=p_b .
\]
Let \(D\) be the cycle obtained by joining \(Q\) with the \(p_j\)-\(p_i\)
subpath of \(P\). Equivalently,
\[
D=
\begin{cases}
	q_1,q_2,\ldots,q_w,p_{j-1},p_{j-2},\ldots,p_i,
	& \text{if } i<j,\\[2mm]
	q_1,q_2,\ldots,q_w,p_{j+1},p_{j+2},\ldots,p_i,
	& \text{if } j<i.
\end{cases}
\]
Since \(i\) and \(j\) are odd, the two edges of \(D\) incident with
\(z=p_b\) do not belong to \(M_v\), whereas all the remaining edges of
\(D\) alternate with respect to \(M_v\). Thus \(D\) is an \(M_v\)-blossom
based at \(z\). In particular,
\[
|E(D)\cap M_v|=\frac{|V(D)|-1}{2}.
\]
Since \(b\) is odd and \(t\) is even, the vertex \(p_{b+1}\) exists and
\[
zp_{b+1}=p_bp_{b+1}\in M_v.
\]
The length-one path
\[
R:=z,p_{b+1}
\]
is therefore \(M_v\)-alternating and its unique edge belongs to \(M_v\).
Also, \(p_{b+1}\notin V(D)\), so
\[
V(D)\cap V(R)=\{z\}.
\]
Hence \(D\) together with \(R\) is an \(M_v\)-perfect flower containing
\(x=q_2\). This contradicts \(x\in PFF(G)\).

\begin{figure}[H]
	
	\begin{center}

		\tikzset{every picture/.style={line width=0.75pt}} 
		
		\begin{tikzpicture}[x=0.75pt,y=0.75pt,yscale=-1,xscale=1]
			
			\draw    (458.63,155.34) -- (459,196.32) ;
			\draw    (458.63,155.34) -- (419.54,143.03) ;
			\draw    (395.75,176.4) -- (419.54,143.03) ;
			\draw    (395.75,176.4) -- (420.14,209.34) ;
			\draw    (459,196.32) -- (420.14,209.34) ;
			\draw    (459,196.32) ;
			\draw    (420.14,209.34) ;
			\draw    (419.54,143.03) ;
			\draw    (395.75,176.4) ;
			\draw    (395.75,176.4) ;
			\draw    (420.14,209.34) ;
			\draw    (459,196.32) ;
			\draw    (419.54,143.03) ;
			\draw    (395.75,176.4) ;
			\draw    (420.14,209.34) ;
			\draw    (419.54,143.03) ;
			\draw    (458.63,155.34) ;
			\draw    (459,196.32) ;
			\draw    (395.75,176.4) ;
			\draw [line width=0.75]    (420.14,209.34) ;
			\draw [shift={(420.14,209.34)}, rotate = 0] [color={rgb, 255:red, 0; green, 0; blue, 0 }  ][fill={rgb, 255:red, 0; green, 0; blue, 0 }  ][line width=0.75]      (0, 0) circle [x radius= 3.02, y radius= 3.02]   ;
			\draw [line width=0.75]    (395.75,176.4) ;
			\draw [shift={(395.75,176.4)}, rotate = 0] [color={rgb, 255:red, 0; green, 0; blue, 0 }  ][fill={rgb, 255:red, 0; green, 0; blue, 0 }  ][line width=0.75]      (0, 0) circle [x radius= 3.02, y radius= 3.02]   ;
			\draw [line width=0.75]    (419.54,143.03) ;
			\draw [shift={(419.54,143.03)}, rotate = 0] [color={rgb, 255:red, 0; green, 0; blue, 0 }  ][fill={rgb, 255:red, 0; green, 0; blue, 0 }  ][line width=0.75]      (0, 0) circle [x radius= 3.02, y radius= 3.02]   ;
			\draw [line width=0.75]    (458.63,155.34) ;
			\draw [shift={(458.63,155.34)}, rotate = 0] [color={rgb, 255:red, 0; green, 0; blue, 0 }  ][fill={rgb, 255:red, 0; green, 0; blue, 0 }  ][line width=0.75]      (0, 0) circle [x radius= 3.02, y radius= 3.02]   ;
			\draw [line width=0.75]    (459,196.32) ;
			\draw [shift={(459,196.32)}, rotate = 0] [color={rgb, 255:red, 0; green, 0; blue, 0 }  ][fill={rgb, 255:red, 0; green, 0; blue, 0 }  ][line width=0.75]      (0, 0) circle [x radius= 3.02, y radius= 3.02]   ;
			\draw    (347.29,176.4) -- (395.75,176.4) ;
			\draw [shift={(395.75,176.4)}, rotate = 0] [color={rgb, 255:red, 0; green, 0; blue, 0 }  ][fill={rgb, 255:red, 0; green, 0; blue, 0 }  ][line width=0.75]      (0, 0) circle [x radius= 3.02, y radius= 3.02]   ;
			\draw [shift={(347.29,176.4)}, rotate = 0] [color={rgb, 255:red, 0; green, 0; blue, 0 }  ][fill={rgb, 255:red, 0; green, 0; blue, 0 }  ][line width=0.75]      (0, 0) circle [x radius= 3.02, y radius= 3.02]   ;
			\draw    (298.83,176.4) -- (347.29,176.4) ;
			\draw [shift={(347.29,176.4)}, rotate = 0] [color={rgb, 255:red, 0; green, 0; blue, 0 }  ][fill={rgb, 255:red, 0; green, 0; blue, 0 }  ][line width=0.75]      (0, 0) circle [x radius= 3.02, y radius= 3.02]   ;
			\draw [shift={(298.83,176.4)}, rotate = 0] [color={rgb, 255:red, 0; green, 0; blue, 0 }  ][fill={rgb, 255:red, 0; green, 0; blue, 0 }  ][line width=0.75]      (0, 0) circle [x radius= 3.02, y radius= 3.02]   ;
			\draw    (250.38,176.4) -- (298.83,176.4) ;
			\draw [shift={(298.83,176.4)}, rotate = 0] [color={rgb, 255:red, 0; green, 0; blue, 0 }  ][fill={rgb, 255:red, 0; green, 0; blue, 0 }  ][line width=0.75]      (0, 0) circle [x radius= 3.02, y radius= 3.02]   ;
			\draw [shift={(250.38,176.4)}, rotate = 0] [color={rgb, 255:red, 0; green, 0; blue, 0 }  ][fill={rgb, 255:red, 0; green, 0; blue, 0 }  ][line width=0.75]      (0, 0) circle [x radius= 3.02, y radius= 3.02]   ;
			\draw    (201.92,176.4) -- (250.38,176.4) ;
			\draw [shift={(250.38,176.4)}, rotate = 0] [color={rgb, 255:red, 0; green, 0; blue, 0 }  ][fill={rgb, 255:red, 0; green, 0; blue, 0 }  ][line width=0.75]      (0, 0) circle [x radius= 3.02, y radius= 3.02]   ;
			\draw [shift={(201.92,176.4)}, rotate = 0] [color={rgb, 255:red, 0; green, 0; blue, 0 }  ][fill={rgb, 255:red, 0; green, 0; blue, 0 }  ][line width=0.75]      (0, 0) circle [x radius= 3.02, y radius= 3.02]   ;
			\draw    (153.46,176.4) -- (173.81,176.4) -- (201.92,176.4) ;
			\draw [shift={(201.92,176.4)}, rotate = 0] [color={rgb, 255:red, 0; green, 0; blue, 0 }  ][fill={rgb, 255:red, 0; green, 0; blue, 0 }  ][line width=0.75]      (0, 0) circle [x radius= 3.02, y radius= 3.02]   ;
			\draw [shift={(153.46,176.4)}, rotate = 0] [color={rgb, 255:red, 0; green, 0; blue, 0 }  ][fill={rgb, 255:red, 0; green, 0; blue, 0 }  ][line width=0.75]      (0, 0) circle [x radius= 3.02, y radius= 3.02]   ;
			\draw    (298.38,176.4) -- (298.38,129) ;
			\draw [shift={(298.38,129)}, rotate = 270] [color={rgb, 255:red, 0; green, 0; blue, 0 }  ][fill={rgb, 255:red, 0; green, 0; blue, 0 }  ][line width=0.75]      (0, 0) circle [x radius= 3.02, y radius= 3.02]   ;
			\draw [shift={(298.38,176.4)}, rotate = 270] [color={rgb, 255:red, 0; green, 0; blue, 0 }  ][fill={rgb, 255:red, 0; green, 0; blue, 0 }  ][line width=0.75]      (0, 0) circle [x radius= 3.02, y radius= 3.02]   ;
			\draw [color={rgb, 255:red, 255; green, 0; blue, 0 }  ,draw opacity=1 ][line width=1.5]    (153.46,176.4) -- (201.92,176.4) ;
			\draw [color={rgb, 255:red, 255; green, 0; blue, 0 }  ,draw opacity=1 ][line width=1.5]    (250.38,176.4) -- (298.83,176.4) ;
			\draw [color={rgb, 255:red, 255; green, 0; blue, 0 }  ,draw opacity=1 ][line width=1.5]    (347.29,176.4) -- (395.75,176.4) ;
			\draw [color={rgb, 255:red, 255; green, 0; blue, 0 }  ,draw opacity=1 ][line width=1.5]    (419.54,143.03) -- (458.63,155.34) ;
			\draw [color={rgb, 255:red, 255; green, 0; blue, 0 }  ,draw opacity=1 ][line width=1.5]    (420.14,209.34) -- (459,196.32) ;
			\draw [color={rgb, 255:red, 0; green, 0; blue, 255 }  ,draw opacity=1 ][line width=1.5]    (298.38,176.4) -- (298.38,129) ;
			\draw    (153.46,176.4) ;
			\draw [shift={(153.46,176.4)}, rotate = 0] [color={rgb, 255:red, 0; green, 0; blue, 0 }  ][fill={rgb, 255:red, 0; green, 0; blue, 0 }  ][line width=0.75]      (0, 0) circle [x radius= 3.02, y radius= 3.02]   ;
			\draw [shift={(153.46,176.4)}, rotate = 0] [color={rgb, 255:red, 0; green, 0; blue, 0 }  ][fill={rgb, 255:red, 0; green, 0; blue, 0 }  ][line width=0.75]      (0, 0) circle [x radius= 3.02, y radius= 3.02]   ;
			\draw    (201.92,176.4) ;
			\draw [shift={(201.92,176.4)}, rotate = 0] [color={rgb, 255:red, 0; green, 0; blue, 0 }  ][fill={rgb, 255:red, 0; green, 0; blue, 0 }  ][line width=0.75]      (0, 0) circle [x radius= 3.02, y radius= 3.02]   ;
			\draw [shift={(201.92,176.4)}, rotate = 0] [color={rgb, 255:red, 0; green, 0; blue, 0 }  ][fill={rgb, 255:red, 0; green, 0; blue, 0 }  ][line width=0.75]      (0, 0) circle [x radius= 3.02, y radius= 3.02]   ;
			\draw    (420.14,209.34) ;
			\draw [shift={(420.14,209.34)}, rotate = 0] [color={rgb, 255:red, 0; green, 0; blue, 0 }  ][fill={rgb, 255:red, 0; green, 0; blue, 0 }  ][line width=0.75]      (0, 0) circle [x radius= 3.02, y radius= 3.02]   ;
			\draw [shift={(420.14,209.34)}, rotate = 0] [color={rgb, 255:red, 0; green, 0; blue, 0 }  ][fill={rgb, 255:red, 0; green, 0; blue, 0 }  ][line width=0.75]      (0, 0) circle [x radius= 3.02, y radius= 3.02]   ;
			\draw    (459,196.32) ;
			\draw [shift={(459,196.32)}, rotate = 0] [color={rgb, 255:red, 0; green, 0; blue, 0 }  ][fill={rgb, 255:red, 0; green, 0; blue, 0 }  ][line width=0.75]      (0, 0) circle [x radius= 3.02, y radius= 3.02]   ;
			\draw [shift={(459,196.32)}, rotate = 0] [color={rgb, 255:red, 0; green, 0; blue, 0 }  ][fill={rgb, 255:red, 0; green, 0; blue, 0 }  ][line width=0.75]      (0, 0) circle [x radius= 3.02, y radius= 3.02]   ;
			\draw    (458.63,155.34) ;
			\draw [shift={(458.63,155.34)}, rotate = 0] [color={rgb, 255:red, 0; green, 0; blue, 0 }  ][fill={rgb, 255:red, 0; green, 0; blue, 0 }  ][line width=0.75]      (0, 0) circle [x radius= 3.02, y radius= 3.02]   ;
			\draw [shift={(458.63,155.34)}, rotate = 0] [color={rgb, 255:red, 0; green, 0; blue, 0 }  ][fill={rgb, 255:red, 0; green, 0; blue, 0 }  ][line width=0.75]      (0, 0) circle [x radius= 3.02, y radius= 3.02]   ;
			\draw    (419.54,143.03) ;
			\draw [shift={(419.54,143.03)}, rotate = 0] [color={rgb, 255:red, 0; green, 0; blue, 0 }  ][fill={rgb, 255:red, 0; green, 0; blue, 0 }  ][line width=0.75]      (0, 0) circle [x radius= 3.02, y radius= 3.02]   ;
			\draw [shift={(419.54,143.03)}, rotate = 0] [color={rgb, 255:red, 0; green, 0; blue, 0 }  ][fill={rgb, 255:red, 0; green, 0; blue, 0 }  ][line width=0.75]      (0, 0) circle [x radius= 3.02, y radius= 3.02]   ;
			\draw    (395.75,176.4) ;
			\draw [shift={(395.75,176.4)}, rotate = 0] [color={rgb, 255:red, 0; green, 0; blue, 0 }  ][fill={rgb, 255:red, 0; green, 0; blue, 0 }  ][line width=0.75]      (0, 0) circle [x radius= 3.02, y radius= 3.02]   ;
			\draw [shift={(395.75,176.4)}, rotate = 0] [color={rgb, 255:red, 0; green, 0; blue, 0 }  ][fill={rgb, 255:red, 0; green, 0; blue, 0 }  ][line width=0.75]      (0, 0) circle [x radius= 3.02, y radius= 3.02]   ;
			\draw    (347.29,176.4) ;
			\draw [shift={(347.29,176.4)}, rotate = 0] [color={rgb, 255:red, 0; green, 0; blue, 0 }  ][fill={rgb, 255:red, 0; green, 0; blue, 0 }  ][line width=0.75]      (0, 0) circle [x radius= 3.02, y radius= 3.02]   ;
			\draw [shift={(347.29,176.4)}, rotate = 0] [color={rgb, 255:red, 0; green, 0; blue, 0 }  ][fill={rgb, 255:red, 0; green, 0; blue, 0 }  ][line width=0.75]      (0, 0) circle [x radius= 3.02, y radius= 3.02]   ;
			\draw    (298.83,176.4) ;
			\draw [shift={(298.83,176.4)}, rotate = 0] [color={rgb, 255:red, 0; green, 0; blue, 0 }  ][fill={rgb, 255:red, 0; green, 0; blue, 0 }  ][line width=0.75]      (0, 0) circle [x radius= 3.02, y radius= 3.02]   ;
			\draw [shift={(298.83,176.4)}, rotate = 0] [color={rgb, 255:red, 0; green, 0; blue, 0 }  ][fill={rgb, 255:red, 0; green, 0; blue, 0 }  ][line width=0.75]      (0, 0) circle [x radius= 3.02, y radius= 3.02]   ;
			\draw    (250.38,176.4) ;
			\draw [shift={(250.38,176.4)}, rotate = 0] [color={rgb, 255:red, 0; green, 0; blue, 0 }  ][fill={rgb, 255:red, 0; green, 0; blue, 0 }  ][line width=0.75]      (0, 0) circle [x radius= 3.02, y radius= 3.02]   ;
			\draw [shift={(250.38,176.4)}, rotate = 0] [color={rgb, 255:red, 0; green, 0; blue, 0 }  ][fill={rgb, 255:red, 0; green, 0; blue, 0 }  ][line width=0.75]      (0, 0) circle [x radius= 3.02, y radius= 3.02]   ;
			\draw [color={rgb, 255:red, 255; green, 0; blue, 0 }  ,draw opacity=1 ][line width=1.5]    (201.83,116.57) -- (298.38,129) ;
			\draw [color={rgb, 255:red, 0; green, 0; blue, 255 }  ,draw opacity=1 ][line width=1.5]    (201.83,176.4) -- (201.83,116.57) ;
			\draw    (201.83,116.57) ;
			\draw [shift={(201.83,116.57)}, rotate = 0] [color={rgb, 255:red, 0; green, 0; blue, 0 }  ][fill={rgb, 255:red, 0; green, 0; blue, 0 }  ][line width=0.75]      (0, 0) circle [x radius= 3.02, y radius= 3.02]   ;
			\draw [shift={(201.83,116.57)}, rotate = 0] [color={rgb, 255:red, 0; green, 0; blue, 0 }  ][fill={rgb, 255:red, 0; green, 0; blue, 0 }  ][line width=0.75]      (0, 0) circle [x radius= 3.02, y radius= 3.02]   ;
			\draw    (298.38,129) ;
			\draw [shift={(298.38,129)}, rotate = 0] [color={rgb, 255:red, 0; green, 0; blue, 0 }  ][fill={rgb, 255:red, 0; green, 0; blue, 0 }  ][line width=0.75]      (0, 0) circle [x radius= 3.02, y radius= 3.02]   ;
			\draw [shift={(298.38,129)}, rotate = 0] [color={rgb, 255:red, 0; green, 0; blue, 0 }  ][fill={rgb, 255:red, 0; green, 0; blue, 0 }  ][line width=0.75]      (0, 0) circle [x radius= 3.02, y radius= 3.02]   ;
			\draw  [fill={rgb, 255:red, 184; green, 233; blue, 134 }  ,fill opacity=0.15 ][dash pattern={on 4.5pt off 4.5pt}] (144,175.57) .. controls (144,160.57) and (183,167.57) .. (189,164.57) .. controls (195,161.57) and (186,116.57) .. (196,108.57) .. controls (206,100.57) and (240,108.57) .. (259,112.57) .. controls (278,116.57) and (310,111.57) .. (311,132.57) .. controls (312,153.57) and (312.18,164.6) .. (312,174.57) .. controls (311.82,184.54) and (313,196.57) .. (301,199.57) .. controls (289,202.57) and (300,186.57) .. (273,186.57) .. controls (246,186.57) and (197.97,184.44) .. (184,183.57) .. controls (170.03,182.7) and (144,190.57) .. (144,175.57) -- cycle ;
			
			\draw (294,183.4) node [anchor=north west][inner sep=0.75pt]  [font=\small]  {$v$};
			\draw (299,147.4) node [anchor=north west][inner sep=0.75pt]  [font=\small]  {$e$};

		\end{tikzpicture}

	\end{center}

\caption{Case 1.2: both vertices lie on $P$.}
	\label{j123oi12j3}
	
\end{figure}
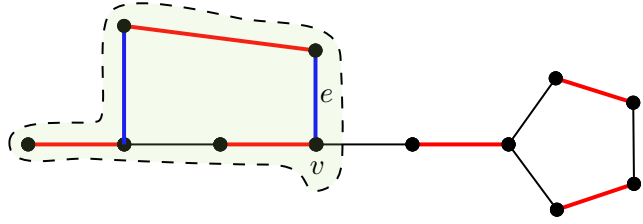

\textbf{Case 1.3.} Assume that \(u\in V(C)\setminus V(P)\), say \(u=c_j\), see
\cref{asdji12i3}. Let \(X\) be the \(u\)-\(p_1\) path in \(C\) whose first edge
belongs to \(M_v\). Such a path exists because exactly one of the two edges
of \(C\) incident with \(u\) belongs to \(M_v\). Its last edge does not
belong to \(M_v\), because both edges of \(C\) incident with \(p_1=c_1\) are
outside \(M_v\).

Let \(P[p_1,p_i]\) denote the \(p_1\)-\(p_i\) subpath of \(P\), possibly of
length zero if \(i=1\). Consider the cycle
\[
D:=X\,P[p_1,p_i]\,Q,
\]
where the paths are concatenated in the natural way. At the vertex
\(v=p_i\), the two incident edges of \(D\) do not belong to \(M_v\): one is
the first edge \(q_1q_2\) of \(Q\), and the other is either the last edge of
\(P[p_1,p_i]\), when \(i>1\), or the last edge of \(X\), when \(i=1\).
All other edges of \(D\) alternate with respect to \(M_v\). Hence \(D\) is
an \(M_v\)-blossom based at \(v\), and
\[
|E(D)\cap M_v|=\frac{|V(D)|-1}{2}.
\]
The length-one path
\[
R:=v,M_v(v)
\]
has its unique edge in \(M_v\). Moreover, \(M_v(v)\notin V(D)\), and hence
\[
V(D)\cap V(R)=\{v\}.
\]
Thus \(D\) together with \(R\) is an \(M_v\)-perfect flower containing
\(x=q_2\). This contradicts \(x\in PFF(G)\).

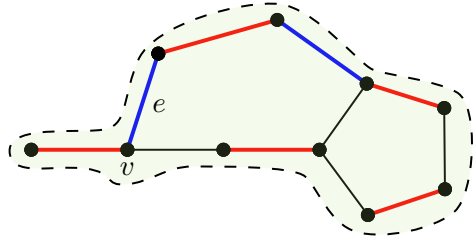
\begin{figure}[H]
	
	\begin{center}

		\tikzset{every picture/.style={line width=0.75pt}} 
		
		\begin{tikzpicture}[x=0.75pt,y=0.75pt,yscale=-1,xscale=1]
			
			\draw [color={rgb, 255:red, 0; green, 0; blue, 255 }  ,draw opacity=1 ][line width=1.5]    (203.83,142.4) -- (219.43,93.86) ;
			\draw [color={rgb, 255:red, 255; green, 0; blue, 0 }  ,draw opacity=1 ][line width=1.5]    (300.75,142.4) -- (252.29,142.4) ;
			\draw [color={rgb, 255:red, 0; green, 0; blue, 255 }  ,draw opacity=1 ][line width=1.5]    (324.54,109.03) -- (279.43,76.86) ;
			\draw    (219.43,93.86) ;
			\draw [shift={(219.43,93.86)}, rotate = 0] [color={rgb, 255:red, 0; green, 0; blue, 0 }  ][fill={rgb, 255:red, 0; green, 0; blue, 0 }  ][line width=0.75]      (0, 0) circle [x radius= 3.02, y radius= 3.02]   ;
			\draw [shift={(219.43,93.86)}, rotate = 0] [color={rgb, 255:red, 0; green, 0; blue, 0 }  ][fill={rgb, 255:red, 0; green, 0; blue, 0 }  ][line width=0.75]      (0, 0) circle [x radius= 3.02, y radius= 3.02]   ;
			\draw    (363.63,121.34) -- (364,162.32) ;
			\draw    (363.63,121.34) -- (324.54,109.03) ;
			\draw    (300.75,142.4) -- (324.54,109.03) ;
			\draw    (300.75,142.4) -- (325.14,175.34) ;
			\draw    (364,162.32) -- (325.14,175.34) ;
			\draw    (364,162.32) ;
			\draw    (325.14,175.34) ;
			\draw    (324.54,109.03) ;
			\draw    (300.75,142.4) ;
			\draw    (300.75,142.4) ;
			\draw    (325.14,175.34) ;
			\draw    (364,162.32) ;
			\draw    (324.54,109.03) ;
			\draw    (300.75,142.4) ;
			\draw    (325.14,175.34) ;
			\draw    (324.54,109.03) ;
			\draw    (363.63,121.34) ;
			\draw    (364,162.32) ;
			\draw    (300.75,142.4) ;
			\draw [line width=0.75]    (325.14,175.34) ;
			\draw [shift={(325.14,175.34)}, rotate = 0] [color={rgb, 255:red, 0; green, 0; blue, 0 }  ][fill={rgb, 255:red, 0; green, 0; blue, 0 }  ][line width=0.75]      (0, 0) circle [x radius= 3.02, y radius= 3.02]   ;
			\draw [line width=0.75]    (300.75,142.4) ;
			\draw [shift={(300.75,142.4)}, rotate = 0] [color={rgb, 255:red, 0; green, 0; blue, 0 }  ][fill={rgb, 255:red, 0; green, 0; blue, 0 }  ][line width=0.75]      (0, 0) circle [x radius= 3.02, y radius= 3.02]   ;
			\draw [line width=0.75]    (324.54,109.03) ;
			\draw [shift={(324.54,109.03)}, rotate = 0] [color={rgb, 255:red, 0; green, 0; blue, 0 }  ][fill={rgb, 255:red, 0; green, 0; blue, 0 }  ][line width=0.75]      (0, 0) circle [x radius= 3.02, y radius= 3.02]   ;
			\draw [line width=0.75]    (363.63,121.34) ;
			\draw [shift={(363.63,121.34)}, rotate = 0] [color={rgb, 255:red, 0; green, 0; blue, 0 }  ][fill={rgb, 255:red, 0; green, 0; blue, 0 }  ][line width=0.75]      (0, 0) circle [x radius= 3.02, y radius= 3.02]   ;
			\draw [line width=0.75]    (364,162.32) ;
			\draw [shift={(364,162.32)}, rotate = 0] [color={rgb, 255:red, 0; green, 0; blue, 0 }  ][fill={rgb, 255:red, 0; green, 0; blue, 0 }  ][line width=0.75]      (0, 0) circle [x radius= 3.02, y radius= 3.02]   ;
			\draw    (203.83,142.4) -- (252.29,142.4) ;
			\draw [shift={(252.29,142.4)}, rotate = 0] [color={rgb, 255:red, 0; green, 0; blue, 0 }  ][fill={rgb, 255:red, 0; green, 0; blue, 0 }  ][line width=0.75]      (0, 0) circle [x radius= 3.02, y radius= 3.02]   ;
			\draw [shift={(203.83,142.4)}, rotate = 0] [color={rgb, 255:red, 0; green, 0; blue, 0 }  ][fill={rgb, 255:red, 0; green, 0; blue, 0 }  ][line width=0.75]      (0, 0) circle [x radius= 3.02, y radius= 3.02]   ;
			\draw    (155.38,142.4) -- (203.83,142.4) ;
			\draw [shift={(203.83,142.4)}, rotate = 0] [color={rgb, 255:red, 0; green, 0; blue, 0 }  ][fill={rgb, 255:red, 0; green, 0; blue, 0 }  ][line width=0.75]      (0, 0) circle [x radius= 3.02, y radius= 3.02]   ;
			\draw [shift={(155.38,142.4)}, rotate = 0] [color={rgb, 255:red, 0; green, 0; blue, 0 }  ][fill={rgb, 255:red, 0; green, 0; blue, 0 }  ][line width=0.75]      (0, 0) circle [x radius= 3.02, y radius= 3.02]   ;
			\draw [color={rgb, 255:red, 255; green, 0; blue, 0 }  ,draw opacity=1 ][line width=1.5]    (155.38,142.4) -- (203.83,142.4) ;
			\draw [color={rgb, 255:red, 255; green, 0; blue, 0 }  ,draw opacity=1 ][line width=1.5]    (325.14,175.34) -- (364,162.32) ;
			\draw [color={rgb, 255:red, 255; green, 0; blue, 0 }  ,draw opacity=1 ][line width=1.5]    (279.43,76.86) -- (219.43,93.86) ;
			\draw [color={rgb, 255:red, 255; green, 0; blue, 0 }  ,draw opacity=1 ][line width=1.5]    (363.63,121.34) -- (324.54,109.03) ;
			\draw    (325.14,175.34) ;
			\draw [shift={(325.14,175.34)}, rotate = 0] [color={rgb, 255:red, 0; green, 0; blue, 0 }  ][fill={rgb, 255:red, 0; green, 0; blue, 0 }  ][line width=0.75]      (0, 0) circle [x radius= 3.02, y radius= 3.02]   ;
			\draw [shift={(325.14,175.34)}, rotate = 0] [color={rgb, 255:red, 0; green, 0; blue, 0 }  ][fill={rgb, 255:red, 0; green, 0; blue, 0 }  ][line width=0.75]      (0, 0) circle [x radius= 3.02, y radius= 3.02]   ;
			\draw    (364,162.32) ;
			\draw [shift={(364,162.32)}, rotate = 0] [color={rgb, 255:red, 0; green, 0; blue, 0 }  ][fill={rgb, 255:red, 0; green, 0; blue, 0 }  ][line width=0.75]      (0, 0) circle [x radius= 3.02, y radius= 3.02]   ;
			\draw [shift={(364,162.32)}, rotate = 0] [color={rgb, 255:red, 0; green, 0; blue, 0 }  ][fill={rgb, 255:red, 0; green, 0; blue, 0 }  ][line width=0.75]      (0, 0) circle [x radius= 3.02, y radius= 3.02]   ;
			\draw    (363.63,121.34) ;
			\draw [shift={(363.63,121.34)}, rotate = 0] [color={rgb, 255:red, 0; green, 0; blue, 0 }  ][fill={rgb, 255:red, 0; green, 0; blue, 0 }  ][line width=0.75]      (0, 0) circle [x radius= 3.02, y radius= 3.02]   ;
			\draw [shift={(363.63,121.34)}, rotate = 0] [color={rgb, 255:red, 0; green, 0; blue, 0 }  ][fill={rgb, 255:red, 0; green, 0; blue, 0 }  ][line width=0.75]      (0, 0) circle [x radius= 3.02, y radius= 3.02]   ;
			\draw    (324.54,109.03) ;
			\draw [shift={(324.54,109.03)}, rotate = 0] [color={rgb, 255:red, 0; green, 0; blue, 0 }  ][fill={rgb, 255:red, 0; green, 0; blue, 0 }  ][line width=0.75]      (0, 0) circle [x radius= 3.02, y radius= 3.02]   ;
			\draw [shift={(324.54,109.03)}, rotate = 0] [color={rgb, 255:red, 0; green, 0; blue, 0 }  ][fill={rgb, 255:red, 0; green, 0; blue, 0 }  ][line width=0.75]      (0, 0) circle [x radius= 3.02, y radius= 3.02]   ;
			\draw    (300.75,142.4) ;
			\draw [shift={(300.75,142.4)}, rotate = 0] [color={rgb, 255:red, 0; green, 0; blue, 0 }  ][fill={rgb, 255:red, 0; green, 0; blue, 0 }  ][line width=0.75]      (0, 0) circle [x radius= 3.02, y radius= 3.02]   ;
			\draw [shift={(300.75,142.4)}, rotate = 0] [color={rgb, 255:red, 0; green, 0; blue, 0 }  ][fill={rgb, 255:red, 0; green, 0; blue, 0 }  ][line width=0.75]      (0, 0) circle [x radius= 3.02, y radius= 3.02]   ;
			\draw    (252.29,142.4) ;
			\draw [shift={(252.29,142.4)}, rotate = 0] [color={rgb, 255:red, 0; green, 0; blue, 0 }  ][fill={rgb, 255:red, 0; green, 0; blue, 0 }  ][line width=0.75]      (0, 0) circle [x radius= 3.02, y radius= 3.02]   ;
			\draw [shift={(252.29,142.4)}, rotate = 0] [color={rgb, 255:red, 0; green, 0; blue, 0 }  ][fill={rgb, 255:red, 0; green, 0; blue, 0 }  ][line width=0.75]      (0, 0) circle [x radius= 3.02, y radius= 3.02]   ;
			\draw    (203.83,142.4) ;
			\draw [shift={(203.83,142.4)}, rotate = 0] [color={rgb, 255:red, 0; green, 0; blue, 0 }  ][fill={rgb, 255:red, 0; green, 0; blue, 0 }  ][line width=0.75]      (0, 0) circle [x radius= 3.02, y radius= 3.02]   ;
			\draw [shift={(203.83,142.4)}, rotate = 0] [color={rgb, 255:red, 0; green, 0; blue, 0 }  ][fill={rgb, 255:red, 0; green, 0; blue, 0 }  ][line width=0.75]      (0, 0) circle [x radius= 3.02, y radius= 3.02]   ;
			\draw    (155.38,142.4) ;
			\draw [shift={(155.38,142.4)}, rotate = 0] [color={rgb, 255:red, 0; green, 0; blue, 0 }  ][fill={rgb, 255:red, 0; green, 0; blue, 0 }  ][line width=0.75]      (0, 0) circle [x radius= 3.02, y radius= 3.02]   ;
			\draw [shift={(155.38,142.4)}, rotate = 0] [color={rgb, 255:red, 0; green, 0; blue, 0 }  ][fill={rgb, 255:red, 0; green, 0; blue, 0 }  ][line width=0.75]      (0, 0) circle [x radius= 3.02, y radius= 3.02]   ;
			\draw    (279.43,76.86) ;
			\draw [shift={(279.43,76.86)}, rotate = 0] [color={rgb, 255:red, 0; green, 0; blue, 0 }  ][fill={rgb, 255:red, 0; green, 0; blue, 0 }  ][line width=0.75]      (0, 0) circle [x radius= 3.02, y radius= 3.02]   ;
			\draw [shift={(279.43,76.86)}, rotate = 0] [color={rgb, 255:red, 0; green, 0; blue, 0 }  ][fill={rgb, 255:red, 0; green, 0; blue, 0 }  ][line width=0.75]      (0, 0) circle [x radius= 3.02, y radius= 3.02]   ;
			\draw  [fill={rgb, 255:red, 184; green, 233; blue, 134 }  ,fill opacity=0.15 ][dash pattern={on 4.5pt off 4.5pt}] (209.43,96.86) .. controls (217.43,78.86) and (262.43,66.86) .. (280.43,68.86) .. controls (298.43,70.86) and (320.43,98.86) .. (327.43,101.86) .. controls (334.43,104.86) and (358.43,106.86) .. (370,113.43) .. controls (381.57,120) and (379,164.43) .. (369,173.43) .. controls (359,182.43) and (333.43,187.86) .. (319.43,183.86) .. controls (305.43,179.86) and (299.86,161.29) .. (293.43,155.86) .. controls (287,150.43) and (249.43,150.86) .. (236.43,150.86) .. controls (223.43,150.86) and (215.43,158.86) .. (203.43,159.86) .. controls (191.43,160.86) and (198.43,153.86) .. (182.43,151.86) .. controls (166.43,149.86) and (144.43,156.86) .. (144.43,142.86) .. controls (144.43,128.86) and (182.43,134.86) .. (193.43,131.86) .. controls (204.43,128.86) and (201.43,114.86) .. (209.43,96.86) -- cycle ;
			\draw    (219.43,93.86) ;
			\draw [shift={(219.43,93.86)}, rotate = 0] [color={rgb, 255:red, 0; green, 0; blue, 0 }  ][fill={rgb, 255:red, 0; green, 0; blue, 0 }  ][line width=0.75]      (0, 0) circle [x radius= 3.02, y radius= 3.02]   ;
			\draw [shift={(219.43,93.86)}, rotate = 0] [color={rgb, 255:red, 0; green, 0; blue, 0 }  ][fill={rgb, 255:red, 0; green, 0; blue, 0 }  ][line width=0.75]      (0, 0) circle [x radius= 3.02, y radius= 3.02]   ;
			
			\draw (215,115.4) node [anchor=north west][inner sep=0.75pt]  [font=\small]  {$e$};
			\draw (198.54,147.43) node [anchor=north west][inner sep=0.75pt]  [font=\small]  {$v$};

		\end{tikzpicture}

	\end{center}

\caption{Case 1.3: one vertex lies on $P$ and the other on $C$.}
	\label{asdji12i3}
	
\end{figure}

\textbf{Case 2.} Assume that \(v\in V(P)\), say \(v=p_i\) for some even \(i\).
\cref{123123uyhasd} illustrates the subcase in which \(u\) also lies on \(P\).

Since \(w\geq 4\), the vertex \(q_3\) is defined and lies outside
\(V(P)\cup V(C)\). Consider the path
\[
R:=p_1,p_2,\ldots,p_i(=q_1),q_2,q_3 .
\]
The cycle \(C\) is an \(M_v\)-blossom based at \(p_1=c_1\). The path \(R\)
is \(M_v\)-alternating: its first edge \(p_1p_2\) belongs to \(M_v\), the
edge \(p_{i-1}p_i\) belongs to \(M_v\) because \(i\) is even, the edge
\(q_1q_2\) does not belong to \(M_v\), and the edge \(q_2q_3\) belongs to
\(M_v\). Hence the first and last edges of \(R\) belong to \(M_v\).

Furthermore,
\[
V(C)\cap V(R)=\{p_1\},
\]
because \(V(C)\cap V(P)=\{p_1\}\) and \(q_2,q_3\notin V(P)\cup V(C)\).
Therefore \(C\) together with \(R\) is an \(M_v\)-perfect flower containing
\(x=q_2\). This contradicts \(x\in PFF(G)\).

\begin{figure}[H]
	
	\begin{center}

		\tikzset{every picture/.style={line width=0.75pt}} 
		
		\begin{tikzpicture}[x=0.75pt,y=0.75pt,yscale=-1,xscale=1]
			
			\draw    (458.63,155.34) -- (459,196.32) ;
			\draw    (458.63,155.34) -- (419.54,143.03) ;
			\draw    (395.75,176.4) -- (419.54,143.03) ;
			\draw    (395.75,176.4) -- (420.14,209.34) ;
			\draw    (459,196.32) -- (420.14,209.34) ;
			\draw    (459,196.32) ;
			\draw    (420.14,209.34) ;
			\draw    (419.54,143.03) ;
			\draw    (395.75,176.4) ;
			\draw    (395.75,176.4) ;
			\draw    (420.14,209.34) ;
			\draw    (459,196.32) ;
			\draw    (419.54,143.03) ;
			\draw    (395.75,176.4) ;
			\draw    (420.14,209.34) ;
			\draw    (419.54,143.03) ;
			\draw    (458.63,155.34) ;
			\draw    (459,196.32) ;
			\draw    (395.75,176.4) ;
			\draw [line width=0.75]    (420.14,209.34) ;
			\draw [shift={(420.14,209.34)}, rotate = 0] [color={rgb, 255:red, 0; green, 0; blue, 0 }  ][fill={rgb, 255:red, 0; green, 0; blue, 0 }  ][line width=0.75]      (0, 0) circle [x radius= 3.02, y radius= 3.02]   ;
			\draw [line width=0.75]    (395.75,176.4) ;
			\draw [shift={(395.75,176.4)}, rotate = 0] [color={rgb, 255:red, 0; green, 0; blue, 0 }  ][fill={rgb, 255:red, 0; green, 0; blue, 0 }  ][line width=0.75]      (0, 0) circle [x radius= 3.02, y radius= 3.02]   ;
			\draw [line width=0.75]    (419.54,143.03) ;
			\draw [shift={(419.54,143.03)}, rotate = 0] [color={rgb, 255:red, 0; green, 0; blue, 0 }  ][fill={rgb, 255:red, 0; green, 0; blue, 0 }  ][line width=0.75]      (0, 0) circle [x radius= 3.02, y radius= 3.02]   ;
			\draw [line width=0.75]    (458.63,155.34) ;
			\draw [shift={(458.63,155.34)}, rotate = 0] [color={rgb, 255:red, 0; green, 0; blue, 0 }  ][fill={rgb, 255:red, 0; green, 0; blue, 0 }  ][line width=0.75]      (0, 0) circle [x radius= 3.02, y radius= 3.02]   ;
			\draw [line width=0.75]    (459,196.32) ;
			\draw [shift={(459,196.32)}, rotate = 0] [color={rgb, 255:red, 0; green, 0; blue, 0 }  ][fill={rgb, 255:red, 0; green, 0; blue, 0 }  ][line width=0.75]      (0, 0) circle [x radius= 3.02, y radius= 3.02]   ;
			\draw    (347.29,176.4) -- (395.75,176.4) ;
			\draw [shift={(395.75,176.4)}, rotate = 0] [color={rgb, 255:red, 0; green, 0; blue, 0 }  ][fill={rgb, 255:red, 0; green, 0; blue, 0 }  ][line width=0.75]      (0, 0) circle [x radius= 3.02, y radius= 3.02]   ;
			\draw [shift={(347.29,176.4)}, rotate = 0] [color={rgb, 255:red, 0; green, 0; blue, 0 }  ][fill={rgb, 255:red, 0; green, 0; blue, 0 }  ][line width=0.75]      (0, 0) circle [x radius= 3.02, y radius= 3.02]   ;
			\draw    (298.83,176.4) -- (347.29,176.4) ;
			\draw [shift={(347.29,176.4)}, rotate = 0] [color={rgb, 255:red, 0; green, 0; blue, 0 }  ][fill={rgb, 255:red, 0; green, 0; blue, 0 }  ][line width=0.75]      (0, 0) circle [x radius= 3.02, y radius= 3.02]   ;
			\draw [shift={(298.83,176.4)}, rotate = 0] [color={rgb, 255:red, 0; green, 0; blue, 0 }  ][fill={rgb, 255:red, 0; green, 0; blue, 0 }  ][line width=0.75]      (0, 0) circle [x radius= 3.02, y radius= 3.02]   ;
			\draw    (250.38,176.4) -- (298.83,176.4) ;
			\draw [shift={(298.83,176.4)}, rotate = 0] [color={rgb, 255:red, 0; green, 0; blue, 0 }  ][fill={rgb, 255:red, 0; green, 0; blue, 0 }  ][line width=0.75]      (0, 0) circle [x radius= 3.02, y radius= 3.02]   ;
			\draw [shift={(250.38,176.4)}, rotate = 0] [color={rgb, 255:red, 0; green, 0; blue, 0 }  ][fill={rgb, 255:red, 0; green, 0; blue, 0 }  ][line width=0.75]      (0, 0) circle [x radius= 3.02, y radius= 3.02]   ;
			\draw    (201.92,176.4) -- (250.38,176.4) ;
			\draw [shift={(250.38,176.4)}, rotate = 0] [color={rgb, 255:red, 0; green, 0; blue, 0 }  ][fill={rgb, 255:red, 0; green, 0; blue, 0 }  ][line width=0.75]      (0, 0) circle [x radius= 3.02, y radius= 3.02]   ;
			\draw [shift={(201.92,176.4)}, rotate = 0] [color={rgb, 255:red, 0; green, 0; blue, 0 }  ][fill={rgb, 255:red, 0; green, 0; blue, 0 }  ][line width=0.75]      (0, 0) circle [x radius= 3.02, y radius= 3.02]   ;
			\draw    (153.46,176.4) -- (173.81,176.4) -- (201.92,176.4) ;
			\draw [shift={(201.92,176.4)}, rotate = 0] [color={rgb, 255:red, 0; green, 0; blue, 0 }  ][fill={rgb, 255:red, 0; green, 0; blue, 0 }  ][line width=0.75]      (0, 0) circle [x radius= 3.02, y radius= 3.02]   ;
			\draw [shift={(153.46,176.4)}, rotate = 0] [color={rgb, 255:red, 0; green, 0; blue, 0 }  ][fill={rgb, 255:red, 0; green, 0; blue, 0 }  ][line width=0.75]      (0, 0) circle [x radius= 3.02, y radius= 3.02]   ;
			\draw    (347.38,176.4) -- (347.38,116.57) ;
			\draw [shift={(347.38,116.57)}, rotate = 270] [color={rgb, 255:red, 0; green, 0; blue, 0 }  ][fill={rgb, 255:red, 0; green, 0; blue, 0 }  ][line width=0.75]      (0, 0) circle [x radius= 3.02, y radius= 3.02]   ;
			\draw [shift={(347.38,176.4)}, rotate = 270] [color={rgb, 255:red, 0; green, 0; blue, 0 }  ][fill={rgb, 255:red, 0; green, 0; blue, 0 }  ][line width=0.75]      (0, 0) circle [x radius= 3.02, y radius= 3.02]   ;
			\draw [color={rgb, 255:red, 255; green, 0; blue, 0 }  ,draw opacity=1 ][line width=1.5]    (153.46,176.4) -- (201.92,176.4) ;
			\draw [color={rgb, 255:red, 255; green, 0; blue, 0 }  ,draw opacity=1 ][line width=1.5]    (250.38,176.4) -- (298.83,176.4) ;
			\draw [color={rgb, 255:red, 255; green, 0; blue, 0 }  ,draw opacity=1 ][line width=1.5]    (347.29,176.4) -- (395.75,176.4) ;
			\draw [color={rgb, 255:red, 255; green, 0; blue, 0 }  ,draw opacity=1 ][line width=1.5]    (419.54,143.03) -- (458.63,155.34) ;
			\draw [color={rgb, 255:red, 255; green, 0; blue, 0 }  ,draw opacity=1 ][line width=1.5]    (420.14,209.34) -- (459,196.32) ;
			\draw [color={rgb, 255:red, 0; green, 0; blue, 255 }  ,draw opacity=1 ][line width=1.5]    (347.29,176.4) -- (347.29,116.57) ;
			\draw    (153.46,176.4) ;
			\draw [shift={(153.46,176.4)}, rotate = 0] [color={rgb, 255:red, 0; green, 0; blue, 0 }  ][fill={rgb, 255:red, 0; green, 0; blue, 0 }  ][line width=0.75]      (0, 0) circle [x radius= 3.02, y radius= 3.02]   ;
			\draw [shift={(153.46,176.4)}, rotate = 0] [color={rgb, 255:red, 0; green, 0; blue, 0 }  ][fill={rgb, 255:red, 0; green, 0; blue, 0 }  ][line width=0.75]      (0, 0) circle [x radius= 3.02, y radius= 3.02]   ;
			\draw    (201.92,176.4) ;
			\draw [shift={(201.92,176.4)}, rotate = 0] [color={rgb, 255:red, 0; green, 0; blue, 0 }  ][fill={rgb, 255:red, 0; green, 0; blue, 0 }  ][line width=0.75]      (0, 0) circle [x radius= 3.02, y radius= 3.02]   ;
			\draw [shift={(201.92,176.4)}, rotate = 0] [color={rgb, 255:red, 0; green, 0; blue, 0 }  ][fill={rgb, 255:red, 0; green, 0; blue, 0 }  ][line width=0.75]      (0, 0) circle [x radius= 3.02, y radius= 3.02]   ;
			\draw    (420.14,209.34) ;
			\draw [shift={(420.14,209.34)}, rotate = 0] [color={rgb, 255:red, 0; green, 0; blue, 0 }  ][fill={rgb, 255:red, 0; green, 0; blue, 0 }  ][line width=0.75]      (0, 0) circle [x radius= 3.02, y radius= 3.02]   ;
			\draw [shift={(420.14,209.34)}, rotate = 0] [color={rgb, 255:red, 0; green, 0; blue, 0 }  ][fill={rgb, 255:red, 0; green, 0; blue, 0 }  ][line width=0.75]      (0, 0) circle [x radius= 3.02, y radius= 3.02]   ;
			\draw    (459,196.32) ;
			\draw [shift={(459,196.32)}, rotate = 0] [color={rgb, 255:red, 0; green, 0; blue, 0 }  ][fill={rgb, 255:red, 0; green, 0; blue, 0 }  ][line width=0.75]      (0, 0) circle [x radius= 3.02, y radius= 3.02]   ;
			\draw [shift={(459,196.32)}, rotate = 0] [color={rgb, 255:red, 0; green, 0; blue, 0 }  ][fill={rgb, 255:red, 0; green, 0; blue, 0 }  ][line width=0.75]      (0, 0) circle [x radius= 3.02, y radius= 3.02]   ;
			\draw    (458.63,155.34) ;
			\draw [shift={(458.63,155.34)}, rotate = 0] [color={rgb, 255:red, 0; green, 0; blue, 0 }  ][fill={rgb, 255:red, 0; green, 0; blue, 0 }  ][line width=0.75]      (0, 0) circle [x radius= 3.02, y radius= 3.02]   ;
			\draw [shift={(458.63,155.34)}, rotate = 0] [color={rgb, 255:red, 0; green, 0; blue, 0 }  ][fill={rgb, 255:red, 0; green, 0; blue, 0 }  ][line width=0.75]      (0, 0) circle [x radius= 3.02, y radius= 3.02]   ;
			\draw    (419.54,143.03) ;
			\draw [shift={(419.54,143.03)}, rotate = 0] [color={rgb, 255:red, 0; green, 0; blue, 0 }  ][fill={rgb, 255:red, 0; green, 0; blue, 0 }  ][line width=0.75]      (0, 0) circle [x radius= 3.02, y radius= 3.02]   ;
			\draw [shift={(419.54,143.03)}, rotate = 0] [color={rgb, 255:red, 0; green, 0; blue, 0 }  ][fill={rgb, 255:red, 0; green, 0; blue, 0 }  ][line width=0.75]      (0, 0) circle [x radius= 3.02, y radius= 3.02]   ;
			\draw    (395.75,176.4) ;
			\draw [shift={(395.75,176.4)}, rotate = 0] [color={rgb, 255:red, 0; green, 0; blue, 0 }  ][fill={rgb, 255:red, 0; green, 0; blue, 0 }  ][line width=0.75]      (0, 0) circle [x radius= 3.02, y radius= 3.02]   ;
			\draw [shift={(395.75,176.4)}, rotate = 0] [color={rgb, 255:red, 0; green, 0; blue, 0 }  ][fill={rgb, 255:red, 0; green, 0; blue, 0 }  ][line width=0.75]      (0, 0) circle [x radius= 3.02, y radius= 3.02]   ;
			\draw    (347.29,176.4) ;
			\draw [shift={(347.29,176.4)}, rotate = 0] [color={rgb, 255:red, 0; green, 0; blue, 0 }  ][fill={rgb, 255:red, 0; green, 0; blue, 0 }  ][line width=0.75]      (0, 0) circle [x radius= 3.02, y radius= 3.02]   ;
			\draw [shift={(347.29,176.4)}, rotate = 0] [color={rgb, 255:red, 0; green, 0; blue, 0 }  ][fill={rgb, 255:red, 0; green, 0; blue, 0 }  ][line width=0.75]      (0, 0) circle [x radius= 3.02, y radius= 3.02]   ;
			\draw    (298.83,176.4) ;
			\draw [shift={(298.83,176.4)}, rotate = 0] [color={rgb, 255:red, 0; green, 0; blue, 0 }  ][fill={rgb, 255:red, 0; green, 0; blue, 0 }  ][line width=0.75]      (0, 0) circle [x radius= 3.02, y radius= 3.02]   ;
			\draw [shift={(298.83,176.4)}, rotate = 0] [color={rgb, 255:red, 0; green, 0; blue, 0 }  ][fill={rgb, 255:red, 0; green, 0; blue, 0 }  ][line width=0.75]      (0, 0) circle [x radius= 3.02, y radius= 3.02]   ;
			\draw    (250.38,176.4) ;
			\draw [shift={(250.38,176.4)}, rotate = 0] [color={rgb, 255:red, 0; green, 0; blue, 0 }  ][fill={rgb, 255:red, 0; green, 0; blue, 0 }  ][line width=0.75]      (0, 0) circle [x radius= 3.02, y radius= 3.02]   ;
			\draw [shift={(250.38,176.4)}, rotate = 0] [color={rgb, 255:red, 0; green, 0; blue, 0 }  ][fill={rgb, 255:red, 0; green, 0; blue, 0 }  ][line width=0.75]      (0, 0) circle [x radius= 3.02, y radius= 3.02]   ;
			\draw [color={rgb, 255:red, 255; green, 0; blue, 0 }  ,draw opacity=1 ][line width=1.5]    (250.38,116.57) -- (347.29,116.57) ;
			\draw [color={rgb, 255:red, 0; green, 0; blue, 255 }  ,draw opacity=1 ][line width=1.5]    (250.38,176.4) -- (250.38,116.57) ;
			\draw    (347.38,116.57) ;
			\draw [shift={(347.38,116.57)}, rotate = 0] [color={rgb, 255:red, 0; green, 0; blue, 0 }  ][fill={rgb, 255:red, 0; green, 0; blue, 0 }  ][line width=0.75]      (0, 0) circle [x radius= 3.02, y radius= 3.02]   ;
			\draw [shift={(347.38,116.57)}, rotate = 0] [color={rgb, 255:red, 0; green, 0; blue, 0 }  ][fill={rgb, 255:red, 0; green, 0; blue, 0 }  ][line width=0.75]      (0, 0) circle [x radius= 3.02, y radius= 3.02]   ;
			\draw    (250.38,176.4) ;
			\draw [shift={(250.38,176.4)}, rotate = 0] [color={rgb, 255:red, 0; green, 0; blue, 0 }  ][fill={rgb, 255:red, 0; green, 0; blue, 0 }  ][line width=0.75]      (0, 0) circle [x radius= 3.02, y radius= 3.02]   ;
			\draw [shift={(250.38,176.4)}, rotate = 0] [color={rgb, 255:red, 0; green, 0; blue, 0 }  ][fill={rgb, 255:red, 0; green, 0; blue, 0 }  ][line width=0.75]      (0, 0) circle [x radius= 3.02, y radius= 3.02]   ;
			\draw    (250.38,116.57) ;
			\draw [shift={(250.38,116.57)}, rotate = 0] [color={rgb, 255:red, 0; green, 0; blue, 0 }  ][fill={rgb, 255:red, 0; green, 0; blue, 0 }  ][line width=0.75]      (0, 0) circle [x radius= 3.02, y radius= 3.02]   ;
			\draw [shift={(250.38,116.57)}, rotate = 0] [color={rgb, 255:red, 0; green, 0; blue, 0 }  ][fill={rgb, 255:red, 0; green, 0; blue, 0 }  ][line width=0.75]      (0, 0) circle [x radius= 3.02, y radius= 3.02]   ;
			
			\draw (245,182.4) node [anchor=north west][inner sep=0.75pt]  [font=\small]  {$u$};
			\draw (342,183.4) node [anchor=north west][inner sep=0.75pt]  [font=\small]  {$v$};

		\end{tikzpicture}

	\end{center}

\caption{Case 2: \(v\) lies on \(P\) at an even position; the drawing shows
	the subcase in which \(u\) also lies on \(P\).}
	\label{123123uyhasd}
	
\end{figure}
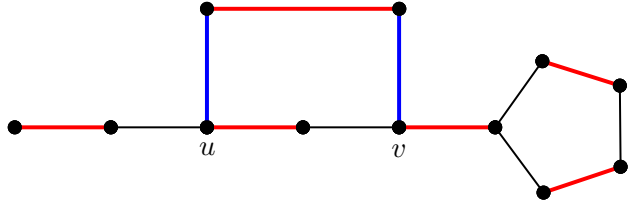

\textbf{Case 3.} Assume that $v\in V(C)-V(P)$, say $v=c_{i}$
for some $i$. According to the position of the vertex $u$, we consider
the following subcases.

$ $

\textbf{Case 3.1.} Assume that \(u\in V(P)\), say \(u=p_j\) for some odd \(j\).
Let \(Y\) be the \(p_1\)-\(v\) path in \(C\) whose last edge belongs to
\(M_v\). Such a path exists because exactly one of the two edges of \(C\)
incident with \(v\) belongs to \(M_v\). Its first edge does not belong to
\(M_v\), since both edges of \(C\) incident with \(p_1=c_1\) are outside
\(M_v\).

Let \(P[u,p_1]\) denote the \(u\)-\(p_1\) subpath of \(P\), possibly of
length zero if \(u=p_1\). Consider the cycle
\[
D:=Y\,Q\,P[u,p_1],
\]
where the paths are concatenated in the natural way. At the vertex
\(u=p_j\), the two incident edges of \(D\) do not belong to \(M_v\): one is
the last edge \(q_{w-1}q_w\) of \(Q\), and the other is either the first edge
of \(P[u,p_1]\), when \(j>1\), or the first edge of \(Y\), when \(j=1\).
All other edges of \(D\) alternate with respect to \(M_v\). Hence \(D\) is
an \(M_v\)-blossom based at \(u\), and
\[
|E(D)\cap M_v|=\frac{|V(D)|-1}{2}.
\]
The length-one path
\[
R:=u,M_v(u)
\]
has its unique edge in \(M_v\). Moreover, \(M_v(u)\notin V(D)\), and hence
\[
V(D)\cap V(R)=\{u\}.
\]
Therefore \(D\) together with \(R\) is an \(M_v\)-perfect flower containing
\(x=q_2\). This contradicts \(x\in PFF(G)\).

$ $

\textbf{Case 3.2.} Assume that \(u\in V(P)\), say \(u=p_j\) for some even \(j\),
see \cref{jasoidji123}. Consider the path
\[
R:=p_1,p_2,\ldots,p_j(=q_w),q_{w-1},q_{w-2},\ldots,q_2 .
\]
The cycle \(C\) is an \(M_v\)-blossom based at \(p_1=c_1\). Since \(j\) is
even, the edge \(p_{j-1}p_j\) belongs to \(M_v\). The next edge
\(q_wq_{w-1}\) does not belong to \(M_v\), and the reversed part of \(Q\)
alternates with respect to \(M_v\), ending with the edge
\(q_3q_2\in M_v\). Therefore \(R\) is \(M_v\)-alternating and its first and
last edges belong to \(M_v\).

Moreover,
\[
V(C)\cap V(R)=\{p_1\},
\]
because \(q_2,q_3,\ldots,q_{w-1}\) lie outside \(V(P)\cup V(C)\). Hence
\(C\) together with \(R\) is an \(M_v\)-perfect flower containing \(x=q_2\).
This contradicts \(x\in PFF(G)\).

\begin{figure}[H]
	
	\begin{center}

	\tikzset{every picture/.style={line width=0.75pt}} 
	
	\begin{tikzpicture}[x=0.75pt,y=0.75pt,yscale=-1,xscale=1]
		
		\draw    (279.63,131.34) -- (280,172.32) ;
		\draw    (279.63,131.34) -- (240.54,119.03) ;
		\draw    (216.75,152.4) -- (240.54,119.03) ;
		\draw    (216.75,152.4) -- (241.14,185.34) ;
		\draw    (280,172.32) -- (241.14,185.34) ;
		\draw    (280,172.32) ;
		\draw    (241.14,185.34) ;
		\draw    (240.54,119.03) ;
		\draw    (216.75,152.4) ;
		\draw    (216.75,152.4) ;
		\draw    (241.14,185.34) ;
		\draw    (280,172.32) ;
		\draw    (240.54,119.03) ;
		\draw    (216.75,152.4) ;
		\draw    (241.14,185.34) ;
		\draw    (240.54,119.03) ;
		\draw    (279.63,131.34) ;
		\draw    (280,172.32) ;
		\draw    (216.75,152.4) ;
		\draw [line width=0.75]    (241.14,185.34) ;
		\draw [shift={(241.14,185.34)}, rotate = 0] [color={rgb, 255:red, 0; green, 0; blue, 0 }  ][fill={rgb, 255:red, 0; green, 0; blue, 0 }  ][line width=0.75]      (0, 0) circle [x radius= 3.02, y radius= 3.02]   ;
		\draw [line width=0.75]    (216.75,152.4) ;
		\draw [shift={(216.75,152.4)}, rotate = 0] [color={rgb, 255:red, 0; green, 0; blue, 0 }  ][fill={rgb, 255:red, 0; green, 0; blue, 0 }  ][line width=0.75]      (0, 0) circle [x radius= 3.02, y radius= 3.02]   ;
		\draw [line width=0.75]    (240.54,119.03) ;
		\draw [shift={(240.54,119.03)}, rotate = 0] [color={rgb, 255:red, 0; green, 0; blue, 0 }  ][fill={rgb, 255:red, 0; green, 0; blue, 0 }  ][line width=0.75]      (0, 0) circle [x radius= 3.02, y radius= 3.02]   ;
		\draw [line width=0.75]    (279.63,131.34) ;
		\draw [shift={(279.63,131.34)}, rotate = 0] [color={rgb, 255:red, 0; green, 0; blue, 0 }  ][fill={rgb, 255:red, 0; green, 0; blue, 0 }  ][line width=0.75]      (0, 0) circle [x radius= 3.02, y radius= 3.02]   ;
		\draw [line width=0.75]    (280,172.32) ;
		\draw [shift={(280,172.32)}, rotate = 0] [color={rgb, 255:red, 0; green, 0; blue, 0 }  ][fill={rgb, 255:red, 0; green, 0; blue, 0 }  ][line width=0.75]      (0, 0) circle [x radius= 3.02, y radius= 3.02]   ;
		\draw    (168.29,152.4) -- (216.75,152.4) ;
		\draw [shift={(216.75,152.4)}, rotate = 0] [color={rgb, 255:red, 0; green, 0; blue, 0 }  ][fill={rgb, 255:red, 0; green, 0; blue, 0 }  ][line width=0.75]      (0, 0) circle [x radius= 3.02, y radius= 3.02]   ;
		\draw [shift={(168.29,152.4)}, rotate = 0] [color={rgb, 255:red, 0; green, 0; blue, 0 }  ][fill={rgb, 255:red, 0; green, 0; blue, 0 }  ][line width=0.75]      (0, 0) circle [x radius= 3.02, y radius= 3.02]   ;
		\draw    (119.83,152.4) -- (168.29,152.4) ;
		\draw [shift={(168.29,152.4)}, rotate = 0] [color={rgb, 255:red, 0; green, 0; blue, 0 }  ][fill={rgb, 255:red, 0; green, 0; blue, 0 }  ][line width=0.75]      (0, 0) circle [x radius= 3.02, y radius= 3.02]   ;
		\draw [shift={(119.83,152.4)}, rotate = 0] [color={rgb, 255:red, 0; green, 0; blue, 0 }  ][fill={rgb, 255:red, 0; green, 0; blue, 0 }  ][line width=0.75]      (0, 0) circle [x radius= 3.02, y radius= 3.02]   ;
		\draw    (71.38,152.4) -- (119.83,152.4) ;
		\draw [shift={(119.83,152.4)}, rotate = 0] [color={rgb, 255:red, 0; green, 0; blue, 0 }  ][fill={rgb, 255:red, 0; green, 0; blue, 0 }  ][line width=0.75]      (0, 0) circle [x radius= 3.02, y radius= 3.02]   ;
		\draw [shift={(71.38,152.4)}, rotate = 0] [color={rgb, 255:red, 0; green, 0; blue, 0 }  ][fill={rgb, 255:red, 0; green, 0; blue, 0 }  ][line width=0.75]      (0, 0) circle [x radius= 3.02, y radius= 3.02]   ;
		\draw    (168.29,152.4) -- (168.29,105) ;
		\draw [shift={(168.29,105)}, rotate = 270] [color={rgb, 255:red, 0; green, 0; blue, 0 }  ][fill={rgb, 255:red, 0; green, 0; blue, 0 }  ][line width=0.75]      (0, 0) circle [x radius= 3.02, y radius= 3.02]   ;
		\draw [shift={(168.29,152.4)}, rotate = 270] [color={rgb, 255:red, 0; green, 0; blue, 0 }  ][fill={rgb, 255:red, 0; green, 0; blue, 0 }  ][line width=0.75]      (0, 0) circle [x radius= 3.02, y radius= 3.02]   ;
		\draw [color={rgb, 255:red, 255; green, 0; blue, 0 }  ,draw opacity=1 ][line width=1.5]    (71.38,152.4) -- (119.83,152.4) ;
		\draw [color={rgb, 255:red, 255; green, 0; blue, 0 }  ,draw opacity=1 ][line width=1.5]    (168.29,152.4) -- (216.75,152.4) ;
		\draw [color={rgb, 255:red, 255; green, 0; blue, 0 }  ,draw opacity=1 ][line width=1.5]    (240.54,119.03) -- (279.63,131.34) ;
		\draw [color={rgb, 255:red, 255; green, 0; blue, 0 }  ,draw opacity=1 ][line width=1.5]    (241.14,185.34) -- (280,172.32) ;
		\draw [color={rgb, 255:red, 0; green, 0; blue, 255 }  ,draw opacity=1 ][line width=1.5]    (168.29,152.4) -- (168.29,105) ;
		\draw    (241.14,185.34) ;
		\draw [shift={(241.14,185.34)}, rotate = 0] [color={rgb, 255:red, 0; green, 0; blue, 0 }  ][fill={rgb, 255:red, 0; green, 0; blue, 0 }  ][line width=0.75]      (0, 0) circle [x radius= 3.02, y radius= 3.02]   ;
		\draw [shift={(241.14,185.34)}, rotate = 0] [color={rgb, 255:red, 0; green, 0; blue, 0 }  ][fill={rgb, 255:red, 0; green, 0; blue, 0 }  ][line width=0.75]      (0, 0) circle [x radius= 3.02, y radius= 3.02]   ;
		\draw    (280,172.32) ;
		\draw [shift={(280,172.32)}, rotate = 0] [color={rgb, 255:red, 0; green, 0; blue, 0 }  ][fill={rgb, 255:red, 0; green, 0; blue, 0 }  ][line width=0.75]      (0, 0) circle [x radius= 3.02, y radius= 3.02]   ;
		\draw [shift={(280,172.32)}, rotate = 0] [color={rgb, 255:red, 0; green, 0; blue, 0 }  ][fill={rgb, 255:red, 0; green, 0; blue, 0 }  ][line width=0.75]      (0, 0) circle [x radius= 3.02, y radius= 3.02]   ;
		\draw    (279.63,131.34) ;
		\draw [shift={(279.63,131.34)}, rotate = 0] [color={rgb, 255:red, 0; green, 0; blue, 0 }  ][fill={rgb, 255:red, 0; green, 0; blue, 0 }  ][line width=0.75]      (0, 0) circle [x radius= 3.02, y radius= 3.02]   ;
		\draw [shift={(279.63,131.34)}, rotate = 0] [color={rgb, 255:red, 0; green, 0; blue, 0 }  ][fill={rgb, 255:red, 0; green, 0; blue, 0 }  ][line width=0.75]      (0, 0) circle [x radius= 3.02, y radius= 3.02]   ;
		\draw    (240.54,119.03) ;
		\draw [shift={(240.54,119.03)}, rotate = 0] [color={rgb, 255:red, 0; green, 0; blue, 0 }  ][fill={rgb, 255:red, 0; green, 0; blue, 0 }  ][line width=0.75]      (0, 0) circle [x radius= 3.02, y radius= 3.02]   ;
		\draw [shift={(240.54,119.03)}, rotate = 0] [color={rgb, 255:red, 0; green, 0; blue, 0 }  ][fill={rgb, 255:red, 0; green, 0; blue, 0 }  ][line width=0.75]      (0, 0) circle [x radius= 3.02, y radius= 3.02]   ;
		\draw    (216.75,152.4) ;
		\draw [shift={(216.75,152.4)}, rotate = 0] [color={rgb, 255:red, 0; green, 0; blue, 0 }  ][fill={rgb, 255:red, 0; green, 0; blue, 0 }  ][line width=0.75]      (0, 0) circle [x radius= 3.02, y radius= 3.02]   ;
		\draw [shift={(216.75,152.4)}, rotate = 0] [color={rgb, 255:red, 0; green, 0; blue, 0 }  ][fill={rgb, 255:red, 0; green, 0; blue, 0 }  ][line width=0.75]      (0, 0) circle [x radius= 3.02, y radius= 3.02]   ;
		\draw    (168.29,152.4) ;
		\draw [shift={(168.29,152.4)}, rotate = 0] [color={rgb, 255:red, 0; green, 0; blue, 0 }  ][fill={rgb, 255:red, 0; green, 0; blue, 0 }  ][line width=0.75]      (0, 0) circle [x radius= 3.02, y radius= 3.02]   ;
		\draw [shift={(168.29,152.4)}, rotate = 0] [color={rgb, 255:red, 0; green, 0; blue, 0 }  ][fill={rgb, 255:red, 0; green, 0; blue, 0 }  ][line width=0.75]      (0, 0) circle [x radius= 3.02, y radius= 3.02]   ;
		\draw    (119.83,152.4) ;
		\draw [shift={(119.83,152.4)}, rotate = 0] [color={rgb, 255:red, 0; green, 0; blue, 0 }  ][fill={rgb, 255:red, 0; green, 0; blue, 0 }  ][line width=0.75]      (0, 0) circle [x radius= 3.02, y radius= 3.02]   ;
		\draw [shift={(119.83,152.4)}, rotate = 0] [color={rgb, 255:red, 0; green, 0; blue, 0 }  ][fill={rgb, 255:red, 0; green, 0; blue, 0 }  ][line width=0.75]      (0, 0) circle [x radius= 3.02, y radius= 3.02]   ;
		\draw    (71.38,152.4) ;
		\draw [shift={(71.38,152.4)}, rotate = 0] [color={rgb, 255:red, 0; green, 0; blue, 0 }  ][fill={rgb, 255:red, 0; green, 0; blue, 0 }  ][line width=0.75]      (0, 0) circle [x radius= 3.02, y radius= 3.02]   ;
		\draw [shift={(71.38,152.4)}, rotate = 0] [color={rgb, 255:red, 0; green, 0; blue, 0 }  ][fill={rgb, 255:red, 0; green, 0; blue, 0 }  ][line width=0.75]      (0, 0) circle [x radius= 3.02, y radius= 3.02]   ;
		\draw [color={rgb, 255:red, 255; green, 0; blue, 0 }  ,draw opacity=1 ][line width=1.5]    (168.29,105) -- (224.91,78.86) ;
		\draw [color={rgb, 255:red, 0; green, 0; blue, 255 }  ,draw opacity=1 ][line width=1.5]    (224.91,78.86) -- (240.54,119.03) ;
		\draw    (240.54,119.03) ;
		\draw [shift={(240.54,119.03)}, rotate = 0] [color={rgb, 255:red, 0; green, 0; blue, 0 }  ][fill={rgb, 255:red, 0; green, 0; blue, 0 }  ][line width=0.75]      (0, 0) circle [x radius= 3.02, y radius= 3.02]   ;
		\draw [shift={(240.54,119.03)}, rotate = 0] [color={rgb, 255:red, 0; green, 0; blue, 0 }  ][fill={rgb, 255:red, 0; green, 0; blue, 0 }  ][line width=0.75]      (0, 0) circle [x radius= 3.02, y radius= 3.02]   ;
		\draw    (224.91,78.86) ;
		\draw [shift={(224.91,78.86)}, rotate = 0] [color={rgb, 255:red, 0; green, 0; blue, 0 }  ][fill={rgb, 255:red, 0; green, 0; blue, 0 }  ][line width=0.75]      (0, 0) circle [x radius= 3.02, y radius= 3.02]   ;
		\draw [shift={(224.91,78.86)}, rotate = 0] [color={rgb, 255:red, 0; green, 0; blue, 0 }  ][fill={rgb, 255:red, 0; green, 0; blue, 0 }  ][line width=0.75]      (0, 0) circle [x radius= 3.02, y radius= 3.02]   ;
		\draw    (168.29,105) ;
		\draw [shift={(168.29,105)}, rotate = 0] [color={rgb, 255:red, 0; green, 0; blue, 0 }  ][fill={rgb, 255:red, 0; green, 0; blue, 0 }  ][line width=0.75]      (0, 0) circle [x radius= 3.02, y radius= 3.02]   ;
		\draw [shift={(168.29,105)}, rotate = 0] [color={rgb, 255:red, 0; green, 0; blue, 0 }  ][fill={rgb, 255:red, 0; green, 0; blue, 0 }  ][line width=0.75]      (0, 0) circle [x radius= 3.02, y radius= 3.02]   ;
		\draw  [fill={rgb, 255:red, 184; green, 233; blue, 134 }  ,fill opacity=0.15 ][dash pattern={on 4.5pt off 4.5pt}] (151,105.14) .. controls (157,89.14) and (227,56.86) .. (236,74.86) .. controls (245,92.86) and (201,97.86) .. (189,106.86) .. controls (177,115.86) and (182,122.14) .. (184,133.14) .. controls (186,144.14) and (206,138.14) .. (219,120.14) .. controls (232,102.14) and (283,105.14) .. (291,129.14) .. controls (299,153.14) and (297,184.14) .. (266,195.14) .. controls (235,206.14) and (214,194.14) .. (210,173.14) .. controls (206,152.14) and (168.43,168.43) .. (159.43,160.43) .. controls (150.43,152.43) and (145,121.14) .. (151,105.14) -- cycle ;
		
		\draw (238.54,126.43) node [anchor=north west][inner sep=0.75pt]  [font=\small]  {$v$};
		\draw (236,94.4) node [anchor=north west][inner sep=0.75pt]  [font=\small]  {$e$};

	\end{tikzpicture}

	\end{center}

\caption{Case 3.2: one vertex lies on $C-V(P)$ and the other on $P$.}
	\label{jasoidji123}
	
\end{figure}
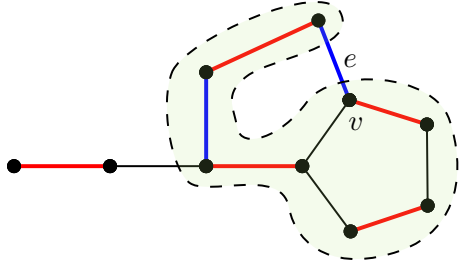

\textbf{Case 3.3.} Assume that \(u\in V(C)\setminus V(P)\). Then both vertices
\(u\) and \(v\) lie in \(V(C)\setminus V(P)\). Write
\[
v=c_i,\qquad u=c_j .
\]
By reversing the cyclic orientation of \(C\), if necessary, we may assume
that
\[
2\leq i<j\leq 2r+1.
\]
This does not interchange \(u\) and \(v\); it only fixes the notation along
the cycle \(C\). Let
\[
X:=c_i,c_{i+1},\ldots,c_j .
\]
Since \(X\) does not contain the base \(c_1\), it is an \(M_v\)-alternating
path. We consider the four possibilities determined by the first and last
edges of \(X\).

$ $

\textbf{Case 3.3.1.}  If the first and last edges of \(X\) belong to \(M_v\), then
\[
C^*:=q_1,q_2,\ldots,q_w(=c_j),c_{j-1},c_{j-2},\ldots,c_i(=q_1)
\]
is an even \(M_v\)-alternating cycle, see \cref{asdioj12ij3}. Therefore
\[
M:=M_v\triangle E(C^*)
\]
is again a perfect matching of \(G\).

Now consider the cycle
\[
D:=c_1,c_2,\ldots,c_i(=q_1),q_2,\ldots,q_w(=c_j),
c_{j+1},\ldots,c_{2r+1},c_1,
\]
where the segment \(c_{j+1},\ldots,c_{2r+1}\) is omitted if \(j=2r+1\).
After rotating the matching along \(C^*\), the cycle \(D\) is an
\(M\)-blossom based at \(c_1=p_1\). Indeed, the edges of \(Q\) have been
switched, the edges of \(C\setminus X\) have not been switched, and the only
two consecutive non-\(M\) edges on \(D\) are the two edges incident with
\(c_1\). Hence
\[
|E(D)\cap M|=\frac{|V(D)|-1}{2}.
\]
The length-one path
\[
R:=p_1,p_2
\]
has its unique edge in \(M\), and
\[
V(D)\cap V(R)=\{p_1\}.
\]
Therefore \(D\) together with \(R\) is an \(M\)-perfect flower containing
\(x=q_2\). This contradicts \(x\in PFF(G)\).

\begin{figure}[H]
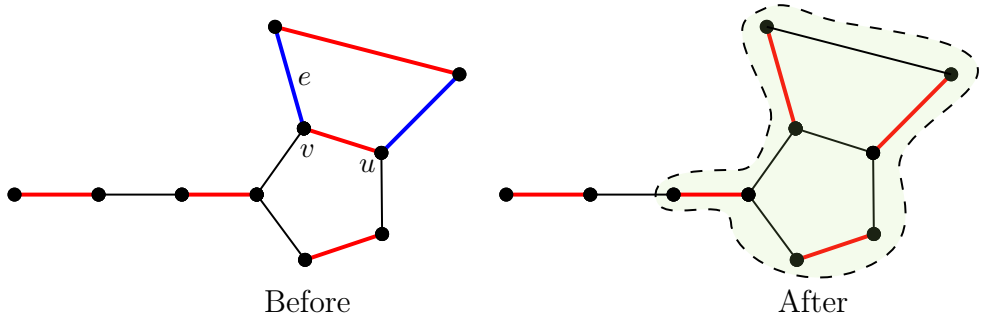

	
	\begin{center}

		\tikzset{every picture/.style={line width=0.75pt}} 
		


	\end{center}

\caption{Case 3.3.1: rotating the matching on the alternating cycle.}
	\label{asdioj12ij3}
	
\end{figure}

\textbf{Case 3.3.2.} If neither the first nor the last edge of \(X\) belongs to
\(M_v\), then consider the cycle
\[
D:=c_1,c_2,\ldots,c_i(=q_1),q_2,\ldots,q_w(=c_j),
c_{j+1},\ldots,c_{2r+1},c_1,
\]
where the segment \(c_{j+1},\ldots,c_{2r+1}\) is omitted if \(j=2r+1\).
The cycle \(D\) is an \(M_v\)-blossom based at \(c_1=p_1\), see \cref{asdoi1h23uoi}.
Indeed, the two edges of \(D\) incident with \(c_1\) do not belong to
\(M_v\), while all remaining edges alternate with respect to \(M_v\). Hence
\[
|E(D)\cap M_v|=\frac{|V(D)|-1}{2}.
\]
The length-one path
\[
R:=p_1,p_2
\]
has its unique edge in \(M_v\), and
\[
V(D)\cap V(R)=\{p_1\}.
\]
Therefore \(D\) together with \(R\) is an \(M_v\)-perfect flower containing
\(x=q_2\). This contradicts \(x\in PFF(G)\).

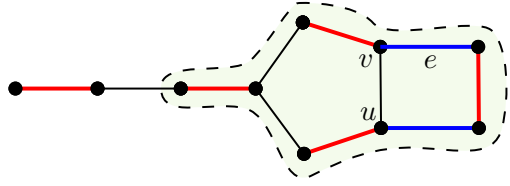
\begin{figure}[H]
	
	\begin{center}

		\tikzset{every picture/.style={line width=0.75pt}} 
		
		\begin{tikzpicture}[x=0.75pt,y=0.75pt,yscale=-1,xscale=1]
			
			\draw  [fill={rgb, 255:red, 184; green, 233; blue, 134 }  ,fill opacity=0.15 ][dash pattern={on 4.5pt off 4.5pt}] (309,128.29) .. controls (313,127.79) and (338,128.29) .. (347,138.29) .. controls (356,148.29) and (399,134.29) .. (407,142.29) .. controls (415,150.29) and (416,175.29) .. (412,191.29) .. controls (408,207.29) and (390,202.29) .. (369,201.29) .. controls (348,200.29) and (332,207.29) .. (316,215.29) .. controls (300,223.29) and (299,185.29) .. (282,181.29) .. controls (265,177.29) and (241,187.71) .. (240,171.71) .. controls (239,155.71) and (276,165.29) .. (285,157.29) .. controls (294,149.29) and (305,128.79) .. (309,128.29) -- cycle ;
			\draw    (350.63,150.34) -- (351,191.32) ;
			\draw    (350.63,150.34) -- (311.54,138.03) ;
			\draw    (287.75,171.4) -- (311.54,138.03) ;
			\draw    (287.75,171.4) -- (312.14,204.34) ;
			\draw    (351,191.32) -- (312.14,204.34) ;
			\draw    (351,191.32) ;
			\draw    (312.14,204.34) ;
			\draw    (311.54,138.03) ;
			\draw    (287.75,171.4) ;
			\draw    (287.75,171.4) ;
			\draw    (312.14,204.34) ;
			\draw    (351,191.32) ;
			\draw    (311.54,138.03) ;
			\draw    (287.75,171.4) ;
			\draw    (312.14,204.34) ;
			\draw    (311.54,138.03) ;
			\draw    (350.63,150.34) ;
			\draw    (351,191.32) ;
			\draw    (287.75,171.4) ;
			\draw [line width=0.75]    (312.14,204.34) ;
			\draw [shift={(312.14,204.34)}, rotate = 0] [color={rgb, 255:red, 0; green, 0; blue, 0 }  ][fill={rgb, 255:red, 0; green, 0; blue, 0 }  ][line width=0.75]      (0, 0) circle [x radius= 3.02, y radius= 3.02]   ;
			\draw [line width=0.75]    (287.75,171.4) ;
			\draw [shift={(287.75,171.4)}, rotate = 0] [color={rgb, 255:red, 0; green, 0; blue, 0 }  ][fill={rgb, 255:red, 0; green, 0; blue, 0 }  ][line width=0.75]      (0, 0) circle [x radius= 3.02, y radius= 3.02]   ;
			\draw [line width=0.75]    (311.54,138.03) ;
			\draw [shift={(311.54,138.03)}, rotate = 0] [color={rgb, 255:red, 0; green, 0; blue, 0 }  ][fill={rgb, 255:red, 0; green, 0; blue, 0 }  ][line width=0.75]      (0, 0) circle [x radius= 3.02, y radius= 3.02]   ;
			\draw [line width=0.75]    (350.63,150.34) ;
			\draw [shift={(350.63,150.34)}, rotate = 0] [color={rgb, 255:red, 0; green, 0; blue, 0 }  ][fill={rgb, 255:red, 0; green, 0; blue, 0 }  ][line width=0.75]      (0, 0) circle [x radius= 3.02, y radius= 3.02]   ;
			\draw [line width=0.75]    (351,191.32) ;
			\draw [shift={(351,191.32)}, rotate = 0] [color={rgb, 255:red, 0; green, 0; blue, 0 }  ][fill={rgb, 255:red, 0; green, 0; blue, 0 }  ][line width=0.75]      (0, 0) circle [x radius= 3.02, y radius= 3.02]   ;
			\draw    (208,171.4) -- (250,171.4) ;
			\draw [shift={(250,171.4)}, rotate = 0] [color={rgb, 255:red, 0; green, 0; blue, 0 }  ][fill={rgb, 255:red, 0; green, 0; blue, 0 }  ][line width=0.75]      (0, 0) circle [x radius= 3.02, y radius= 3.02]   ;
			\draw [shift={(208,171.4)}, rotate = 0] [color={rgb, 255:red, 0; green, 0; blue, 0 }  ][fill={rgb, 255:red, 0; green, 0; blue, 0 }  ][line width=0.75]      (0, 0) circle [x radius= 3.02, y radius= 3.02]   ;
			\draw [color={rgb, 255:red, 255; green, 0; blue, 0 }  ,draw opacity=1 ][line width=1.5]    (166.6,171.4) -- (208,171.4) ;
			\draw [color={rgb, 255:red, 255; green, 0; blue, 0 }  ,draw opacity=1 ][line width=1.5]    (250,171.4) -- (287.75,171.4) ;
			\draw [color={rgb, 255:red, 255; green, 0; blue, 0 }  ,draw opacity=1 ][line width=1.5]    (311.54,138.03) -- (350.63,150.34) ;
			\draw [color={rgb, 255:red, 255; green, 0; blue, 0 }  ,draw opacity=1 ][line width=1.5]    (312.14,204.34) -- (351,191.32) ;
			\draw [color={rgb, 255:red, 0; green, 0; blue, 255 }  ,draw opacity=1 ][line width=1.5]    (351,191.32) -- (400.37,191.32) ;
			\draw    (312.14,204.34) ;
			\draw [shift={(312.14,204.34)}, rotate = 0] [color={rgb, 255:red, 0; green, 0; blue, 0 }  ][fill={rgb, 255:red, 0; green, 0; blue, 0 }  ][line width=0.75]      (0, 0) circle [x radius= 3.02, y radius= 3.02]   ;
			\draw [shift={(312.14,204.34)}, rotate = 0] [color={rgb, 255:red, 0; green, 0; blue, 0 }  ][fill={rgb, 255:red, 0; green, 0; blue, 0 }  ][line width=0.75]      (0, 0) circle [x radius= 3.02, y radius= 3.02]   ;
			\draw    (351,191.32) ;
			\draw [shift={(351,191.32)}, rotate = 0] [color={rgb, 255:red, 0; green, 0; blue, 0 }  ][fill={rgb, 255:red, 0; green, 0; blue, 0 }  ][line width=0.75]      (0, 0) circle [x radius= 3.02, y radius= 3.02]   ;
			\draw [shift={(351,191.32)}, rotate = 0] [color={rgb, 255:red, 0; green, 0; blue, 0 }  ][fill={rgb, 255:red, 0; green, 0; blue, 0 }  ][line width=0.75]      (0, 0) circle [x radius= 3.02, y radius= 3.02]   ;
			\draw    (350.63,150.34) ;
			\draw [shift={(350.63,150.34)}, rotate = 0] [color={rgb, 255:red, 0; green, 0; blue, 0 }  ][fill={rgb, 255:red, 0; green, 0; blue, 0 }  ][line width=0.75]      (0, 0) circle [x radius= 3.02, y radius= 3.02]   ;
			\draw [shift={(350.63,150.34)}, rotate = 0] [color={rgb, 255:red, 0; green, 0; blue, 0 }  ][fill={rgb, 255:red, 0; green, 0; blue, 0 }  ][line width=0.75]      (0, 0) circle [x radius= 3.02, y radius= 3.02]   ;
			\draw    (311.54,138.03) ;
			\draw [shift={(311.54,138.03)}, rotate = 0] [color={rgb, 255:red, 0; green, 0; blue, 0 }  ][fill={rgb, 255:red, 0; green, 0; blue, 0 }  ][line width=0.75]      (0, 0) circle [x radius= 3.02, y radius= 3.02]   ;
			\draw [shift={(311.54,138.03)}, rotate = 0] [color={rgb, 255:red, 0; green, 0; blue, 0 }  ][fill={rgb, 255:red, 0; green, 0; blue, 0 }  ][line width=0.75]      (0, 0) circle [x radius= 3.02, y radius= 3.02]   ;
			\draw    (287.75,171.4) ;
			\draw [shift={(287.75,171.4)}, rotate = 0] [color={rgb, 255:red, 0; green, 0; blue, 0 }  ][fill={rgb, 255:red, 0; green, 0; blue, 0 }  ][line width=0.75]      (0, 0) circle [x radius= 3.02, y radius= 3.02]   ;
			\draw [shift={(287.75,171.4)}, rotate = 0] [color={rgb, 255:red, 0; green, 0; blue, 0 }  ][fill={rgb, 255:red, 0; green, 0; blue, 0 }  ][line width=0.75]      (0, 0) circle [x radius= 3.02, y radius= 3.02]   ;
			\draw    (250,171.4) ;
			\draw [shift={(250,171.4)}, rotate = 0] [color={rgb, 255:red, 0; green, 0; blue, 0 }  ][fill={rgb, 255:red, 0; green, 0; blue, 0 }  ][line width=0.75]      (0, 0) circle [x radius= 3.02, y radius= 3.02]   ;
			\draw [shift={(250,171.4)}, rotate = 0] [color={rgb, 255:red, 0; green, 0; blue, 0 }  ][fill={rgb, 255:red, 0; green, 0; blue, 0 }  ][line width=0.75]      (0, 0) circle [x radius= 3.02, y radius= 3.02]   ;
			\draw    (166.6,171.4) ;
			\draw [shift={(166.6,171.4)}, rotate = 0] [color={rgb, 255:red, 0; green, 0; blue, 0 }  ][fill={rgb, 255:red, 0; green, 0; blue, 0 }  ][line width=0.75]      (0, 0) circle [x radius= 3.02, y radius= 3.02]   ;
			\draw [shift={(166.6,171.4)}, rotate = 0] [color={rgb, 255:red, 0; green, 0; blue, 0 }  ][fill={rgb, 255:red, 0; green, 0; blue, 0 }  ][line width=0.75]      (0, 0) circle [x radius= 3.02, y radius= 3.02]   ;
			\draw    (208,171.4) ;
			\draw [shift={(208,171.4)}, rotate = 0] [color={rgb, 255:red, 0; green, 0; blue, 0 }  ][fill={rgb, 255:red, 0; green, 0; blue, 0 }  ][line width=0.75]      (0, 0) circle [x radius= 3.02, y radius= 3.02]   ;
			\draw [shift={(208,171.4)}, rotate = 0] [color={rgb, 255:red, 0; green, 0; blue, 0 }  ][fill={rgb, 255:red, 0; green, 0; blue, 0 }  ][line width=0.75]      (0, 0) circle [x radius= 3.02, y radius= 3.02]   ;
			\draw [color={rgb, 255:red, 255; green, 0; blue, 0 }  ,draw opacity=1 ][line width=1.5]    (400,150.34) -- (400.37,191.32) ;
			\draw [color={rgb, 255:red, 0; green, 0; blue, 255 }  ,draw opacity=1 ][line width=1.5]    (400,150.34) -- (350.63,150.34) ;
			\draw    (311.54,138.03) ;
			\draw [shift={(311.54,138.03)}, rotate = 0] [color={rgb, 255:red, 0; green, 0; blue, 0 }  ][fill={rgb, 255:red, 0; green, 0; blue, 0 }  ][line width=0.75]      (0, 0) circle [x radius= 3.02, y radius= 3.02]   ;
			\draw [shift={(311.54,138.03)}, rotate = 0] [color={rgb, 255:red, 0; green, 0; blue, 0 }  ][fill={rgb, 255:red, 0; green, 0; blue, 0 }  ][line width=0.75]      (0, 0) circle [x radius= 3.02, y radius= 3.02]   ;
			\draw    (400,150.34) ;
			\draw [shift={(400,150.34)}, rotate = 0] [color={rgb, 255:red, 0; green, 0; blue, 0 }  ][fill={rgb, 255:red, 0; green, 0; blue, 0 }  ][line width=0.75]      (0, 0) circle [x radius= 3.02, y radius= 3.02]   ;
			\draw [shift={(400,150.34)}, rotate = 0] [color={rgb, 255:red, 0; green, 0; blue, 0 }  ][fill={rgb, 255:red, 0; green, 0; blue, 0 }  ][line width=0.75]      (0, 0) circle [x radius= 3.02, y radius= 3.02]   ;
			\draw    (400.37,191.32) ;
			\draw [shift={(400.37,191.32)}, rotate = 0] [color={rgb, 255:red, 0; green, 0; blue, 0 }  ][fill={rgb, 255:red, 0; green, 0; blue, 0 }  ][line width=0.75]      (0, 0) circle [x radius= 3.02, y radius= 3.02]   ;
			\draw [shift={(400.37,191.32)}, rotate = 0] [color={rgb, 255:red, 0; green, 0; blue, 0 }  ][fill={rgb, 255:red, 0; green, 0; blue, 0 }  ][line width=0.75]      (0, 0) circle [x radius= 3.02, y radius= 3.02]   ;
			
			\draw (338.08,152.59) node [anchor=north west][inner sep=0.75pt]  [font=\small]  {$v$};
			\draw (371,153.4) node [anchor=north west][inner sep=0.75pt]  [font=\small]  {$e$};
			\draw (339,179.4) node [anchor=north west][inner sep=0.75pt]  [font=\small]  {$u$};

		\end{tikzpicture}

	\end{center}

\caption{Case 3.3.2: both end edges of $X$ are outside $M_v$.}
	\label{asdoi1h23uoi}
	
\end{figure}

\textbf{Case 3.3.3.} Assume that the first edge of \(X\) belongs to \(M_v\), and
the last edge of \(X\) does not belong to \(M_v\). Consider the cycle
\[
D:=q_1,q_2,\ldots,q_w(=c_j),c_{j-1},c_{j-2},\ldots,c_i(=q_1).
\]
Then \(D\) is an \(M_v\)-blossom based at \(u=c_j\), see \cref{asd123iu132i}. Indeed,
the last edge \(q_{w-1}q_w\) of \(Q\) and the last edge \(c_{j-1}c_j\) of
\(X\) both lie outside \(M_v\), while all remaining edges of \(D\) alternate
with respect to \(M_v\). Hence
\[
|E(D)\cap M_v|=\frac{|V(D)|-1}{2}.
\]
Since the last edge of \(X\) does not belong to \(M_v\), the edge
\(c_jc_{j+1}=uM_v(u)\) belongs to \(M_v\). Thus the length-one path
\[
R:=c_j,c_{j+1}
\]
has its unique edge in \(M_v\), and
\[
V(D)\cap V(R)=\{c_j\}.
\]
Therefore \(D\) together with \(R\) is an \(M_v\)-perfect flower containing
\(x=q_2\). This contradicts \(x\in PFF(G)\).

Case 3.3.4. Assume that the first edge of \(X\) does not belong to \(M_v\),
and the last edge of \(X\) belongs to \(M_v\). Consider again the cycle
\[
D:=q_1,q_2,\ldots,q_w(=c_j),c_{j-1},c_{j-2},\ldots,c_i(=q_1).
\]
Then \(D\) is an \(M_v\)-blossom based at \(v=c_i\). Indeed, the first edge
\(q_1q_2\) of \(Q\) and the first edge \(c_ic_{i+1}\) of \(X\) both lie
outside \(M_v\), while all remaining edges of \(D\) alternate with respect
to \(M_v\). Hence
\[
|E(D)\cap M_v|=\frac{|V(D)|-1}{2}.
\]
Since the first edge of \(X\) does not belong to \(M_v\), the edge
\(c_ic_{i-1}=vM_v(v)\) belongs to \(M_v\). Thus the length-one path
\[
R:=c_i,c_{i-1}
\]
has its unique edge in \(M_v\), and
\[
V(D)\cap V(R)=\{c_i\}.
\]
Therefore \(D\) together with \(R\) is an \(M_v\)-perfect flower containing
\(x=q_2\). This contradicts \(x\in PFF(G)\).

\begin{figure}[H]
	
	\begin{center}

		\tikzset{every picture/.style={line width=0.75pt}} 
		
		\begin{tikzpicture}[x=0.75pt,y=0.75pt,yscale=-1,xscale=1]
			
			\draw  [fill={rgb, 255:red, 184; green, 233; blue, 134 }  ,fill opacity=0.15 ][dash pattern={on 4.5pt off 4.5pt}] (339,185.71) .. controls (348,180.71) and (343,174.71) .. (340,163.71) .. controls (337,152.71) and (297.63,151.07) .. (300,139.71) .. controls (302.37,128.36) and (318,116.71) .. (331,106.71) .. controls (344,96.71) and (372,81.71) .. (380,93.71) .. controls (388,105.71) and (404,143.71) .. (403,153.71) .. controls (402,163.71) and (362,200.71) .. (347,206.71) .. controls (332,212.71) and (305,219.71) .. (304,203.71) .. controls (303,187.71) and (330,190.71) .. (339,185.71) -- cycle ;
			\draw    (350.63,150.34) -- (351,191.32) ;
			\draw    (350.63,150.34) -- (311.54,138.03) ;
			\draw    (287.75,171.4) -- (311.54,138.03) ;
			\draw    (287.75,171.4) -- (312.14,204.34) ;
			\draw    (351,191.32) -- (312.14,204.34) ;
			\draw    (351,191.32) ;
			\draw    (312.14,204.34) ;
			\draw    (311.54,138.03) ;
			\draw    (287.75,171.4) ;
			\draw    (287.75,171.4) ;
			\draw    (312.14,204.34) ;
			\draw    (351,191.32) ;
			\draw    (311.54,138.03) ;
			\draw    (287.75,171.4) ;
			\draw    (312.14,204.34) ;
			\draw    (311.54,138.03) ;
			\draw    (350.63,150.34) ;
			\draw    (351,191.32) ;
			\draw    (287.75,171.4) ;
			\draw [line width=0.75]    (312.14,204.34) ;
			\draw [shift={(312.14,204.34)}, rotate = 0] [color={rgb, 255:red, 0; green, 0; blue, 0 }  ][fill={rgb, 255:red, 0; green, 0; blue, 0 }  ][line width=0.75]      (0, 0) circle [x radius= 3.02, y radius= 3.02]   ;
			\draw [line width=0.75]    (287.75,171.4) ;
			\draw [shift={(287.75,171.4)}, rotate = 0] [color={rgb, 255:red, 0; green, 0; blue, 0 }  ][fill={rgb, 255:red, 0; green, 0; blue, 0 }  ][line width=0.75]      (0, 0) circle [x radius= 3.02, y radius= 3.02]   ;
			\draw [line width=0.75]    (311.54,138.03) ;
			\draw [shift={(311.54,138.03)}, rotate = 0] [color={rgb, 255:red, 0; green, 0; blue, 0 }  ][fill={rgb, 255:red, 0; green, 0; blue, 0 }  ][line width=0.75]      (0, 0) circle [x radius= 3.02, y radius= 3.02]   ;
			\draw [line width=0.75]    (350.63,150.34) ;
			\draw [shift={(350.63,150.34)}, rotate = 0] [color={rgb, 255:red, 0; green, 0; blue, 0 }  ][fill={rgb, 255:red, 0; green, 0; blue, 0 }  ][line width=0.75]      (0, 0) circle [x radius= 3.02, y radius= 3.02]   ;
			\draw [line width=0.75]    (351,191.32) ;
			\draw [shift={(351,191.32)}, rotate = 0] [color={rgb, 255:red, 0; green, 0; blue, 0 }  ][fill={rgb, 255:red, 0; green, 0; blue, 0 }  ][line width=0.75]      (0, 0) circle [x radius= 3.02, y radius= 3.02]   ;
			\draw    (208,171.4) -- (250,171.4) ;
			\draw [shift={(250,171.4)}, rotate = 0] [color={rgb, 255:red, 0; green, 0; blue, 0 }  ][fill={rgb, 255:red, 0; green, 0; blue, 0 }  ][line width=0.75]      (0, 0) circle [x radius= 3.02, y radius= 3.02]   ;
			\draw [shift={(208,171.4)}, rotate = 0] [color={rgb, 255:red, 0; green, 0; blue, 0 }  ][fill={rgb, 255:red, 0; green, 0; blue, 0 }  ][line width=0.75]      (0, 0) circle [x radius= 3.02, y radius= 3.02]   ;
			\draw [color={rgb, 255:red, 255; green, 0; blue, 0 }  ,draw opacity=1 ][line width=1.5]    (166.6,171.4) -- (208,171.4) ;
			\draw [color={rgb, 255:red, 255; green, 0; blue, 0 }  ,draw opacity=1 ][line width=1.5]    (250,171.4) -- (287.75,171.4) ;
			\draw [color={rgb, 255:red, 255; green, 0; blue, 0 }  ,draw opacity=1 ][line width=1.5]    (311.54,138.03) -- (350.63,150.34) ;
			\draw [color={rgb, 255:red, 255; green, 0; blue, 0 }  ,draw opacity=1 ][line width=1.5]    (312.14,204.34) -- (351,191.32) ;
			\draw [color={rgb, 255:red, 0; green, 0; blue, 255 }  ,draw opacity=1 ][line width=1.5]    (351,191.32) -- (390.37,151.7) ;
			\draw    (312.14,204.34) ;
			\draw [shift={(312.14,204.34)}, rotate = 0] [color={rgb, 255:red, 0; green, 0; blue, 0 }  ][fill={rgb, 255:red, 0; green, 0; blue, 0 }  ][line width=0.75]      (0, 0) circle [x radius= 3.02, y radius= 3.02]   ;
			\draw [shift={(312.14,204.34)}, rotate = 0] [color={rgb, 255:red, 0; green, 0; blue, 0 }  ][fill={rgb, 255:red, 0; green, 0; blue, 0 }  ][line width=0.75]      (0, 0) circle [x radius= 3.02, y radius= 3.02]   ;
			\draw    (351,191.32) ;
			\draw [shift={(351,191.32)}, rotate = 0] [color={rgb, 255:red, 0; green, 0; blue, 0 }  ][fill={rgb, 255:red, 0; green, 0; blue, 0 }  ][line width=0.75]      (0, 0) circle [x radius= 3.02, y radius= 3.02]   ;
			\draw [shift={(351,191.32)}, rotate = 0] [color={rgb, 255:red, 0; green, 0; blue, 0 }  ][fill={rgb, 255:red, 0; green, 0; blue, 0 }  ][line width=0.75]      (0, 0) circle [x radius= 3.02, y radius= 3.02]   ;
			\draw    (350.63,150.34) ;
			\draw [shift={(350.63,150.34)}, rotate = 0] [color={rgb, 255:red, 0; green, 0; blue, 0 }  ][fill={rgb, 255:red, 0; green, 0; blue, 0 }  ][line width=0.75]      (0, 0) circle [x radius= 3.02, y radius= 3.02]   ;
			\draw [shift={(350.63,150.34)}, rotate = 0] [color={rgb, 255:red, 0; green, 0; blue, 0 }  ][fill={rgb, 255:red, 0; green, 0; blue, 0 }  ][line width=0.75]      (0, 0) circle [x radius= 3.02, y radius= 3.02]   ;
			\draw    (311.54,138.03) ;
			\draw [shift={(311.54,138.03)}, rotate = 0] [color={rgb, 255:red, 0; green, 0; blue, 0 }  ][fill={rgb, 255:red, 0; green, 0; blue, 0 }  ][line width=0.75]      (0, 0) circle [x radius= 3.02, y radius= 3.02]   ;
			\draw [shift={(311.54,138.03)}, rotate = 0] [color={rgb, 255:red, 0; green, 0; blue, 0 }  ][fill={rgb, 255:red, 0; green, 0; blue, 0 }  ][line width=0.75]      (0, 0) circle [x radius= 3.02, y radius= 3.02]   ;
			\draw    (287.75,171.4) ;
			\draw [shift={(287.75,171.4)}, rotate = 0] [color={rgb, 255:red, 0; green, 0; blue, 0 }  ][fill={rgb, 255:red, 0; green, 0; blue, 0 }  ][line width=0.75]      (0, 0) circle [x radius= 3.02, y radius= 3.02]   ;
			\draw [shift={(287.75,171.4)}, rotate = 0] [color={rgb, 255:red, 0; green, 0; blue, 0 }  ][fill={rgb, 255:red, 0; green, 0; blue, 0 }  ][line width=0.75]      (0, 0) circle [x radius= 3.02, y radius= 3.02]   ;
			\draw    (250,171.4) ;
			\draw [shift={(250,171.4)}, rotate = 0] [color={rgb, 255:red, 0; green, 0; blue, 0 }  ][fill={rgb, 255:red, 0; green, 0; blue, 0 }  ][line width=0.75]      (0, 0) circle [x radius= 3.02, y radius= 3.02]   ;
			\draw [shift={(250,171.4)}, rotate = 0] [color={rgb, 255:red, 0; green, 0; blue, 0 }  ][fill={rgb, 255:red, 0; green, 0; blue, 0 }  ][line width=0.75]      (0, 0) circle [x radius= 3.02, y radius= 3.02]   ;
			\draw    (166.6,171.4) ;
			\draw [shift={(166.6,171.4)}, rotate = 0] [color={rgb, 255:red, 0; green, 0; blue, 0 }  ][fill={rgb, 255:red, 0; green, 0; blue, 0 }  ][line width=0.75]      (0, 0) circle [x radius= 3.02, y radius= 3.02]   ;
			\draw [shift={(166.6,171.4)}, rotate = 0] [color={rgb, 255:red, 0; green, 0; blue, 0 }  ][fill={rgb, 255:red, 0; green, 0; blue, 0 }  ][line width=0.75]      (0, 0) circle [x radius= 3.02, y radius= 3.02]   ;
			\draw    (208,171.4) ;
			\draw [shift={(208,171.4)}, rotate = 0] [color={rgb, 255:red, 0; green, 0; blue, 0 }  ][fill={rgb, 255:red, 0; green, 0; blue, 0 }  ][line width=0.75]      (0, 0) circle [x radius= 3.02, y radius= 3.02]   ;
			\draw [shift={(208,171.4)}, rotate = 0] [color={rgb, 255:red, 0; green, 0; blue, 0 }  ][fill={rgb, 255:red, 0; green, 0; blue, 0 }  ][line width=0.75]      (0, 0) circle [x radius= 3.02, y radius= 3.02]   ;
			\draw [color={rgb, 255:red, 255; green, 0; blue, 0 }  ,draw opacity=1 ][line width=1.5]    (370,99.86) -- (390.37,151.7) ;
			\draw [color={rgb, 255:red, 0; green, 0; blue, 255 }  ,draw opacity=1 ][line width=1.5]    (370,99.86) -- (311.54,138.03) ;
			\draw    (311.54,138.03) ;
			\draw [shift={(311.54,138.03)}, rotate = 0] [color={rgb, 255:red, 0; green, 0; blue, 0 }  ][fill={rgb, 255:red, 0; green, 0; blue, 0 }  ][line width=0.75]      (0, 0) circle [x radius= 3.02, y radius= 3.02]   ;
			\draw [shift={(311.54,138.03)}, rotate = 0] [color={rgb, 255:red, 0; green, 0; blue, 0 }  ][fill={rgb, 255:red, 0; green, 0; blue, 0 }  ][line width=0.75]      (0, 0) circle [x radius= 3.02, y radius= 3.02]   ;
			\draw    (370,99.86) ;
			\draw [shift={(370,99.86)}, rotate = 0] [color={rgb, 255:red, 0; green, 0; blue, 0 }  ][fill={rgb, 255:red, 0; green, 0; blue, 0 }  ][line width=0.75]      (0, 0) circle [x radius= 3.02, y radius= 3.02]   ;
			\draw [shift={(370,99.86)}, rotate = 0] [color={rgb, 255:red, 0; green, 0; blue, 0 }  ][fill={rgb, 255:red, 0; green, 0; blue, 0 }  ][line width=0.75]      (0, 0) circle [x radius= 3.02, y radius= 3.02]   ;
			\draw    (390.37,151.7) ;
			\draw [shift={(390.37,151.7)}, rotate = 0] [color={rgb, 255:red, 0; green, 0; blue, 0 }  ][fill={rgb, 255:red, 0; green, 0; blue, 0 }  ][line width=0.75]      (0, 0) circle [x radius= 3.02, y radius= 3.02]   ;
			\draw [shift={(390.37,151.7)}, rotate = 0] [color={rgb, 255:red, 0; green, 0; blue, 0 }  ][fill={rgb, 255:red, 0; green, 0; blue, 0 }  ][line width=0.75]      (0, 0) circle [x radius= 3.02, y radius= 3.02]   ;
			
			\draw (320,131.4) node [anchor=north west][inner sep=0.75pt]  [font=\small]  {$v$};
			\draw (341,118.4) node [anchor=north west][inner sep=0.75pt]  [font=\small]  {$e$};
			\draw (350,173.4) node [anchor=north west][inner sep=0.75pt]  [font=\small]  {$u$};

		\end{tikzpicture}

	\end{center}

\caption{Case 3.3.3: the first edge of $X$ belongs to $M_v$, but the last does not.}
	\label{asd123iu132i}
	
\end{figure}

The four subcases in Case~3.3 exhaust all possibilities for the first and
last edges of \(X\). Together with Cases~1, 2, 3.1, and 3.2, this covers all
possible positions of \(v\) and \(u\) in \(V(P)\cup V(C)\), without
interchanging their roles. In every case we have obtained a perfect matching
\(M\) and an \(M\)-perfect flower containing \(x=q_2\). This contradicts
\(x\in PFF(G)\). Therefore no such edge \(e\) exists, and
\[
E(PFF(G),PF(G))\cap E(H)=\emptyset.\qedhere
\]
\end{proof}

Then, from \cref{largo} we obtain the main result
of this work.

\begin{theorem}\label{dettt}
	Let $G$ be a K\H{o}nig-Egerváry graph. Then
	\begin{eqnarray*}
		\det(G) & = & \det\left(G[PF(G)]\right)\det\left(G[PFF(G)]\right),\\
		\perm G & = & \perm{G[PF(G)]}\perm{G[PFF(G)]}.
	\end{eqnarray*}
\end{theorem}

\begin{proof}
	Set
	\[
	G_1:=G[PF(G)]
	\qquad\text{and}\qquad
	G_2:=G[PFF(G)].
	\]
	
	If $\Sa{G}=\emptyset$, then by Theorem 3.1,
	\[
	\det(G)=0
	\qquad\text{and}\qquad
	\perm{G}=0.
	\]
	Moreover, at least one of the graphs $G_1$ and $G_2$ has no Sachs subgraphs.
	Indeed, if there existed
	\[
	H_1\in \Sa{G_1}
	\qquad\text{and}\qquad
	H_2\in \Sa{G_2},
	\]
	then $H_1\cup H_2$ would be a Sachs subgraph of $G$, a contradiction.
	Hence at least one of $\det(G_1),\det(G_2)$ is zero, and at least one of
	$\perm{G_1},\perm{G_2}$ is zero. Therefore,
	\[
	\det(G)=\det(G_1)\det(G_2)
	\qquad\text{and}\qquad
	\perm{G}=\perm{G_1}\perm{G_2}.
	\]
	
	Assume now that $\Sa{G}\neq\emptyset$. Since every Sachs subgraph is spanning,
	we have $\prk{G}=|G|$. Thus, by \cref{asoudh1iuh23},
	\[
	|G|=\prk{G}=2\mu(G),
	\]
	and therefore $G$ has a perfect matching.
	
	Let $H\in \Sa{G}$. By \cref{asoidj1oi23},
	\[
	E(PF(G),PFF(G))\cap E(H)=\emptyset.
	\]
	Hence every connected component of $H$ lies entirely in $PF(G)$ or entirely
	in $PFF(G)$. Therefore
	\[
	H=H_1\cup H_2,
	\]
	where
	\[
	H_1:=H[PF(G)]
	\qquad\text{and}\qquad
	H_2:=H[PFF(G)].
	\]
	Clearly,
	\[
	H_1\in \Sa{G_1}
	\qquad\text{and}\qquad
	H_2\in \Sa{G_2}.
	\]
	
	Conversely, if
	\[
	H_1\in \Sa{G_1}
	\qquad\text{and}\qquad
	H_2\in \Sa{G_2},
	\]
	then $H_1\cup H_2\in \Sa{G}$. Thus
	\[
	H\mapsto (H_1,H_2)
	\]
	is a bijection between $\Sa{G}$ and $\Sa{G_1}\times \Sa{G_2}$.
	
	Since the connected components of $H$ are precisely the connected components
	of $H_1$ together with those of $H_2$, we have
	\[
	k(H)=k(H_1)+k(H_2)
	\qquad\text{and}\qquad
	m(H)=m(H_1)+m(H_2).
	\]
	Hence
	\begin{eqnarray*}
		(-1)^{k(H)}2^{m(H)}
		&=&
		(-1)^{k(H_1)+k(H_2)}2^{m(H_1)+m(H_2)} \\
		&=&
		(-1)^{k(H_1)}2^{m(H_1)}
		(-1)^{k(H_2)}2^{m(H_2)},
	\end{eqnarray*}
	and also
	\[
	2^{m(H)}=2^{m(H_1)}2^{m(H_2)}.
	\]
	
	Using \cref{harary} and the above bijection, we obtain
	\begin{eqnarray*}
		\det(G)
		&=&
		\sum_{H\in \Sa{G}}(-1)^{k(H)}2^{m(H)} \\
		&=&
		\sum_{H_1\in \Sa{G_1}}
		\sum_{H_2\in \Sa{G_2}}
		(-1)^{k(H_1)}2^{m(H_1)}
		(-1)^{k(H_2)}2^{m(H_2)} \\
		&=&
		\left(
		\sum_{H_1\in \Sa{G_1}}
		(-1)^{k(H_1)}2^{m(H_1)}
		\right)
		\left(
		\sum_{H_2\in \Sa{G_2}}
		(-1)^{k(H_2)}2^{m(H_2)}
		\right) \\
		&=&
		\det(G_1)\det(G_2).
	\end{eqnarray*}
	Similarly,
	\begin{eqnarray*}
		\perm{G}
		&=&
		\sum_{H\in \Sa{G}}2^{m(H)} \\
		&=&
		\sum_{H_1\in \Sa{G_1}}
		\sum_{H_2\in \Sa{G_2}}
		2^{m(H_1)}2^{m(H_2)} \\
		&=&
		\left(
		\sum_{H_1\in \Sa{G_1}}2^{m(H_1)}
		\right)
		\left(
		\sum_{H_2\in \Sa{G_2}}2^{m(H_2)}
		\right) \\
		&=&
		\perm{G_1}\perm{G_2}.\qedhere
	\end{eqnarray*}
\end{proof}

By \cref{asokpjdoas} and \cref{largo}, one has the
additivity of $\mu(G)$ according to the perfect flower decomposition.

\begin{corollary}
	Let $G$ be a K\H{o}nig-Egerváry graph with a perfect matching. Then
	\[
	\mu(G)=\mu(G[PF(G)])+\mu(G[PFF(G)]).
	\]
\end{corollary}

\begin{proof}
	Let $M_1$ and $M_2$ be maximum matchings of $G[PF(G)]$ and $G[PFF(G)]$,
	respectively. Since $PF(G)$ and $PFF(G)$ are disjoint, $M_1\cup M_2$ is a
	matching of $G$. Hence
	\[
	\mu(G)\geq \mu(G[PF(G)])+\mu(G[PFF(G)]).
	\]
	
	On the other hand, let $M$ be a perfect matching of $G$. Since $G[M]\in\Sa{G}$,
	\cref{largo} yields
	\[
	E(PF(G),PFF(G))\cap M=\emptyset.
	\]
	Therefore,
	\[
	M=\bigl(M\cap E(G[PF(G)])\bigr)\cup \bigl(M\cap E(G[PFF(G)])\bigr),
	\]
	and so
	\[
	\mu(G)=|M|
	\leq \mu(G[PF(G)])+\mu(G[PFF(G)]).
	\]
	Thus,
	\[
	\mu(G)=\mu(G[PF(G)])+\mu(G[PFF(G)]).\qedhere
	\]
\end{proof}

To illustrate \cref{dettt}, consider the graph in \cref{asdouih12iu3}.
Here one has $\det\left(G[PFF(G)]\right)=-4$, while $\det\left(G[PF(G)]\right)=9$,
therefore, by \Cref{dettt}, it follows that
\[
\det\left(G\right)=-4\cdot9=-36.
\]

\begin{figure}[H]
	
	\begin{center}

		\tikzset{every picture/.style={line width=0.75pt}} 
		
		\begin{tikzpicture}[x=0.75pt,y=0.75pt,yscale=-1,xscale=1]
			
			\draw    (175.99,79.13) -- (212.43,84.26) ;
			\draw [shift={(212.43,84.26)}, rotate = 8.01] [color={rgb, 255:red, 0; green, 0; blue, 0 }  ][fill={rgb, 255:red, 0; green, 0; blue, 0 }  ][line width=0.75]      (0, 0) circle [x radius= 3.35, y radius= 3.35]   ;
			\draw [shift={(175.99,79.13)}, rotate = 8.01] [color={rgb, 255:red, 0; green, 0; blue, 0 }  ][fill={rgb, 255:red, 0; green, 0; blue, 0 }  ][line width=0.75]      (0, 0) circle [x radius= 3.35, y radius= 3.35]   ;
			\draw    (212.43,84.26) -- (249.83,106.44) ;
			\draw [shift={(249.83,106.44)}, rotate = 30.67] [color={rgb, 255:red, 0; green, 0; blue, 0 }  ][fill={rgb, 255:red, 0; green, 0; blue, 0 }  ][line width=0.75]      (0, 0) circle [x radius= 3.35, y radius= 3.35]   ;
			\draw [shift={(212.43,84.26)}, rotate = 30.67] [color={rgb, 255:red, 0; green, 0; blue, 0 }  ][fill={rgb, 255:red, 0; green, 0; blue, 0 }  ][line width=0.75]      (0, 0) circle [x radius= 3.35, y radius= 3.35]   ;
			\draw    (249.83,106.44) -- (290.63,103.58) ;
			\draw [shift={(290.63,103.58)}, rotate = 355.99] [color={rgb, 255:red, 0; green, 0; blue, 0 }  ][fill={rgb, 255:red, 0; green, 0; blue, 0 }  ][line width=0.75]      (0, 0) circle [x radius= 3.35, y radius= 3.35]   ;
			\draw [shift={(249.83,106.44)}, rotate = 355.99] [color={rgb, 255:red, 0; green, 0; blue, 0 }  ][fill={rgb, 255:red, 0; green, 0; blue, 0 }  ][line width=0.75]      (0, 0) circle [x radius= 3.35, y radius= 3.35]   ;
			\draw    (290.63,103.58) -- (329.39,82.11) ;
			\draw [shift={(329.39,82.11)}, rotate = 331.03] [color={rgb, 255:red, 0; green, 0; blue, 0 }  ][fill={rgb, 255:red, 0; green, 0; blue, 0 }  ][line width=0.75]      (0, 0) circle [x radius= 3.35, y radius= 3.35]   ;
			\draw [shift={(290.63,103.58)}, rotate = 331.03] [color={rgb, 255:red, 0; green, 0; blue, 0 }  ][fill={rgb, 255:red, 0; green, 0; blue, 0 }  ][line width=0.75]      (0, 0) circle [x radius= 3.35, y radius= 3.35]   ;
			\draw    (329.39,82.11) -- (363.39,86.41) ;
			\draw [shift={(363.39,86.41)}, rotate = 7.2] [color={rgb, 255:red, 0; green, 0; blue, 0 }  ][fill={rgb, 255:red, 0; green, 0; blue, 0 }  ][line width=0.75]      (0, 0) circle [x radius= 3.35, y radius= 3.35]   ;
			\draw [shift={(329.39,82.11)}, rotate = 7.2] [color={rgb, 255:red, 0; green, 0; blue, 0 }  ][fill={rgb, 255:red, 0; green, 0; blue, 0 }  ][line width=0.75]      (0, 0) circle [x radius= 3.35, y radius= 3.35]   ;
			\draw    (363.39,86.41) -- (383.78,123.61) ;
			\draw [shift={(383.78,123.61)}, rotate = 61.26] [color={rgb, 255:red, 0; green, 0; blue, 0 }  ][fill={rgb, 255:red, 0; green, 0; blue, 0 }  ][line width=0.75]      (0, 0) circle [x radius= 3.35, y radius= 3.35]   ;
			\draw [shift={(363.39,86.41)}, rotate = 61.26] [color={rgb, 255:red, 0; green, 0; blue, 0 }  ][fill={rgb, 255:red, 0; green, 0; blue, 0 }  ][line width=0.75]      (0, 0) circle [x radius= 3.35, y radius= 3.35]   ;
			\draw    (383.78,123.61) -- (419.14,147.21) ;
			\draw [shift={(419.14,147.21)}, rotate = 33.73] [color={rgb, 255:red, 0; green, 0; blue, 0 }  ][fill={rgb, 255:red, 0; green, 0; blue, 0 }  ][line width=0.75]      (0, 0) circle [x radius= 3.35, y radius= 3.35]   ;
			\draw [shift={(383.78,123.61)}, rotate = 33.73] [color={rgb, 255:red, 0; green, 0; blue, 0 }  ][fill={rgb, 255:red, 0; green, 0; blue, 0 }  ][line width=0.75]      (0, 0) circle [x radius= 3.35, y radius= 3.35]   ;
			\draw    (419.14,147.21) -- (431.38,181.55) ;
			\draw [shift={(431.38,181.55)}, rotate = 70.38] [color={rgb, 255:red, 0; green, 0; blue, 0 }  ][fill={rgb, 255:red, 0; green, 0; blue, 0 }  ][line width=0.75]      (0, 0) circle [x radius= 3.35, y radius= 3.35]   ;
			\draw [shift={(419.14,147.21)}, rotate = 70.38] [color={rgb, 255:red, 0; green, 0; blue, 0 }  ][fill={rgb, 255:red, 0; green, 0; blue, 0 }  ][line width=0.75]      (0, 0) circle [x radius= 3.35, y radius= 3.35]   ;
			\draw    (431.38,181.55) -- (408.26,202.3) ;
			\draw [shift={(408.26,202.3)}, rotate = 138.1] [color={rgb, 255:red, 0; green, 0; blue, 0 }  ][fill={rgb, 255:red, 0; green, 0; blue, 0 }  ][line width=0.75]      (0, 0) circle [x radius= 3.35, y radius= 3.35]   ;
			\draw [shift={(431.38,181.55)}, rotate = 138.1] [color={rgb, 255:red, 0; green, 0; blue, 0 }  ][fill={rgb, 255:red, 0; green, 0; blue, 0 }  ][line width=0.75]      (0, 0) circle [x radius= 3.35, y radius= 3.35]   ;
			\draw    (408.26,202.3) -- (376.31,182.27) ;
			\draw [shift={(376.31,182.27)}, rotate = 212.08] [color={rgb, 255:red, 0; green, 0; blue, 0 }  ][fill={rgb, 255:red, 0; green, 0; blue, 0 }  ][line width=0.75]      (0, 0) circle [x radius= 3.35, y radius= 3.35]   ;
			\draw [shift={(408.26,202.3)}, rotate = 212.08] [color={rgb, 255:red, 0; green, 0; blue, 0 }  ][fill={rgb, 255:red, 0; green, 0; blue, 0 }  ][line width=0.75]      (0, 0) circle [x radius= 3.35, y radius= 3.35]   ;
			\draw    (376.31,182.27) -- (353.19,148.65) ;
			\draw [shift={(353.19,148.65)}, rotate = 235.49] [color={rgb, 255:red, 0; green, 0; blue, 0 }  ][fill={rgb, 255:red, 0; green, 0; blue, 0 }  ][line width=0.75]      (0, 0) circle [x radius= 3.35, y radius= 3.35]   ;
			\draw [shift={(376.31,182.27)}, rotate = 235.49] [color={rgb, 255:red, 0; green, 0; blue, 0 }  ][fill={rgb, 255:red, 0; green, 0; blue, 0 }  ][line width=0.75]      (0, 0) circle [x radius= 3.35, y radius= 3.35]   ;
			\draw    (353.19,148.65) -- (313.07,133.62) ;
			\draw [shift={(313.07,133.62)}, rotate = 200.53] [color={rgb, 255:red, 0; green, 0; blue, 0 }  ][fill={rgb, 255:red, 0; green, 0; blue, 0 }  ][line width=0.75]      (0, 0) circle [x radius= 3.35, y radius= 3.35]   ;
			\draw [shift={(353.19,148.65)}, rotate = 200.53] [color={rgb, 255:red, 0; green, 0; blue, 0 }  ][fill={rgb, 255:red, 0; green, 0; blue, 0 }  ][line width=0.75]      (0, 0) circle [x radius= 3.35, y radius= 3.35]   ;
			\draw    (313.07,133.62) -- (278.39,131.48) ;
			\draw [shift={(278.39,131.48)}, rotate = 183.54] [color={rgb, 255:red, 0; green, 0; blue, 0 }  ][fill={rgb, 255:red, 0; green, 0; blue, 0 }  ][line width=0.75]      (0, 0) circle [x radius= 3.35, y radius= 3.35]   ;
			\draw [shift={(313.07,133.62)}, rotate = 183.54] [color={rgb, 255:red, 0; green, 0; blue, 0 }  ][fill={rgb, 255:red, 0; green, 0; blue, 0 }  ][line width=0.75]      (0, 0) circle [x radius= 3.35, y radius= 3.35]   ;
			\draw    (155.32,107.15) -- (175.99,79.13) ;
			\draw [shift={(175.99,79.13)}, rotate = 306.42] [color={rgb, 255:red, 0; green, 0; blue, 0 }  ][fill={rgb, 255:red, 0; green, 0; blue, 0 }  ][line width=0.75]      (0, 0) circle [x radius= 3.35, y radius= 3.35]   ;
			\draw [shift={(155.32,107.15)}, rotate = 306.42] [color={rgb, 255:red, 0; green, 0; blue, 0 }  ][fill={rgb, 255:red, 0; green, 0; blue, 0 }  ][line width=0.75]      (0, 0) circle [x radius= 3.35, y radius= 3.35]   ;
			\draw    (179.12,125.04) -- (155.32,107.15) ;
			\draw [shift={(155.32,107.15)}, rotate = 216.92] [color={rgb, 255:red, 0; green, 0; blue, 0 }  ][fill={rgb, 255:red, 0; green, 0; blue, 0 }  ][line width=0.75]      (0, 0) circle [x radius= 3.35, y radius= 3.35]   ;
			\draw [shift={(179.12,125.04)}, rotate = 216.92] [color={rgb, 255:red, 0; green, 0; blue, 0 }  ][fill={rgb, 255:red, 0; green, 0; blue, 0 }  ][line width=0.75]      (0, 0) circle [x radius= 3.35, y radius= 3.35]   ;
			\draw    (215.83,120.03) -- (179.12,125.04) ;
			\draw [shift={(179.12,125.04)}, rotate = 172.23] [color={rgb, 255:red, 0; green, 0; blue, 0 }  ][fill={rgb, 255:red, 0; green, 0; blue, 0 }  ][line width=0.75]      (0, 0) circle [x radius= 3.35, y radius= 3.35]   ;
			\draw [shift={(215.83,120.03)}, rotate = 172.23] [color={rgb, 255:red, 0; green, 0; blue, 0 }  ][fill={rgb, 255:red, 0; green, 0; blue, 0 }  ][line width=0.75]      (0, 0) circle [x radius= 3.35, y radius= 3.35]   ;
			\draw    (249.83,106.44) -- (215.83,120.03) ;
			\draw [shift={(215.83,120.03)}, rotate = 158.21] [color={rgb, 255:red, 0; green, 0; blue, 0 }  ][fill={rgb, 255:red, 0; green, 0; blue, 0 }  ][line width=0.75]      (0, 0) circle [x radius= 3.35, y radius= 3.35]   ;
			\draw [shift={(249.83,106.44)}, rotate = 158.21] [color={rgb, 255:red, 0; green, 0; blue, 0 }  ][fill={rgb, 255:red, 0; green, 0; blue, 0 }  ][line width=0.75]      (0, 0) circle [x radius= 3.35, y radius= 3.35]   ;
			\draw    (278.39,131.48) -- (264.79,166.53) ;
			\draw [shift={(264.79,166.53)}, rotate = 111.2] [color={rgb, 255:red, 0; green, 0; blue, 0 }  ][fill={rgb, 255:red, 0; green, 0; blue, 0 }  ][line width=0.75]      (0, 0) circle [x radius= 3.35, y radius= 3.35]   ;
			\draw [shift={(278.39,131.48)}, rotate = 111.2] [color={rgb, 255:red, 0; green, 0; blue, 0 }  ][fill={rgb, 255:red, 0; green, 0; blue, 0 }  ][line width=0.75]      (0, 0) circle [x radius= 3.35, y radius= 3.35]   ;
			\draw    (278.39,131.48) -- (249.83,106.44) ;
			\draw [shift={(249.83,106.44)}, rotate = 221.24] [color={rgb, 255:red, 0; green, 0; blue, 0 }  ][fill={rgb, 255:red, 0; green, 0; blue, 0 }  ][line width=0.75]      (0, 0) circle [x radius= 3.35, y radius= 3.35]   ;
			\draw [shift={(278.39,131.48)}, rotate = 221.24] [color={rgb, 255:red, 0; green, 0; blue, 0 }  ][fill={rgb, 255:red, 0; green, 0; blue, 0 }  ][line width=0.75]      (0, 0) circle [x radius= 3.35, y radius= 3.35]   ;
			\draw    (290.63,103.58) -- (278.39,131.48) ;
			\draw [shift={(278.39,131.48)}, rotate = 113.69] [color={rgb, 255:red, 0; green, 0; blue, 0 }  ][fill={rgb, 255:red, 0; green, 0; blue, 0 }  ][line width=0.75]      (0, 0) circle [x radius= 3.35, y radius= 3.35]   ;
			\draw [shift={(290.63,103.58)}, rotate = 113.69] [color={rgb, 255:red, 0; green, 0; blue, 0 }  ][fill={rgb, 255:red, 0; green, 0; blue, 0 }  ][line width=0.75]      (0, 0) circle [x radius= 3.35, y radius= 3.35]   ;
			\draw    (313.07,133.62) -- (290.63,103.58) ;
			\draw [shift={(290.63,103.58)}, rotate = 233.25] [color={rgb, 255:red, 0; green, 0; blue, 0 }  ][fill={rgb, 255:red, 0; green, 0; blue, 0 }  ][line width=0.75]      (0, 0) circle [x radius= 3.35, y radius= 3.35]   ;
			\draw [shift={(313.07,133.62)}, rotate = 233.25] [color={rgb, 255:red, 0; green, 0; blue, 0 }  ][fill={rgb, 255:red, 0; green, 0; blue, 0 }  ][line width=0.75]      (0, 0) circle [x radius= 3.35, y radius= 3.35]   ;
			\draw    (383.78,123.61) -- (353.19,148.65) ;
			\draw [shift={(353.19,148.65)}, rotate = 140.71] [color={rgb, 255:red, 0; green, 0; blue, 0 }  ][fill={rgb, 255:red, 0; green, 0; blue, 0 }  ][line width=0.75]      (0, 0) circle [x radius= 3.35, y radius= 3.35]   ;
			\draw [shift={(383.78,123.61)}, rotate = 140.71] [color={rgb, 255:red, 0; green, 0; blue, 0 }  ][fill={rgb, 255:red, 0; green, 0; blue, 0 }  ][line width=0.75]      (0, 0) circle [x radius= 3.35, y radius= 3.35]   ;
			\draw  [fill={rgb, 255:red, 80; green, 227; blue, 194 }  ,fill opacity=0.1 ][dash pattern={on 4.5pt off 4.5pt}] (356.79,187.47) .. controls (338.34,175.6) and (337.57,164.84) .. (322.96,158.69) .. controls (308.35,152.53) and (260.67,194.06) .. (252.21,177.14) .. controls (243.75,160.22) and (275.28,98.27) .. (287.58,88.27) .. controls (299.89,78.27) and (315.27,72.89) .. (332.18,69.81) .. controls (349.1,66.74) and (385.24,79.48) .. (390.63,100.24) .. controls (396.01,121.01) and (442.15,141) .. (445.99,162.53) .. controls (449.84,184.06) and (427.54,208.67) .. (406.19,214.72) .. controls (384.84,220.77) and (375.25,199.33) .. (356.79,187.47) -- cycle ;
			\draw  [fill={rgb, 255:red, 80; green, 227; blue, 194 }  ,fill opacity=0.1 ][dash pattern={on 4.5pt off 4.5pt}] (212.22,72.56) .. controls (230.46,76.5) and (276.05,91.79) .. (259.13,113.32) .. controls (242.21,134.85) and (163.78,146.38) .. (151.47,123.31) .. controls (139.17,100.24) and (145.6,89.09) .. (159.93,76.41) .. controls (174.26,63.72) and (193.98,68.62) .. (212.22,72.56) -- cycle ;
			
			\draw (160.36,144.35) node [anchor=north west][inner sep=0.75pt]    {$PFF( G)$};
			\draw (286.2,174.5) node [anchor=north west][inner sep=0.75pt]    {$PF( G)$};

		\end{tikzpicture}

	\end{center}

\caption{An example illustrating the determinant factorization via $PF(G)$ and $PFF(G)$.}
	\label{asdouih12iu3}
	
\end{figure}
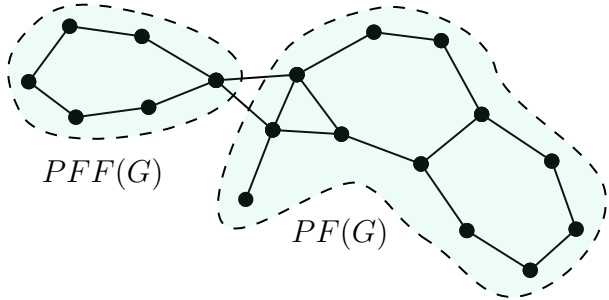

\noindent Therefore, by \cref{dettt}, studying when a K\H{o}nig-Egerváry
graph is unimodular reduces to studying when the graphs $G[PFF(G)]$
and $G[PF(G)]$ are unimodular. This motivates the following problems.

	\section{Open Problems}
	
	It is known that every K\H{o}nig-Egerváry graph with a unique matching is unimodular;
	however, the converse is not true. Indeed, unimodularity
	is only enough to guarantee the existence of a perfect matching.
	For example, this occurs in the graph obtained by considering two cycles
	$C_{4}$ with one common edge. Consequently, the study of
	unimodularity in K\H{o}nig--Egerváry graphs can be reduced to the analysis
	of the following two problems. When $G$ is a K\H{o}nig-Egerváry
	graph with a perfect matching, then by \cref{safe}, both
	graphs
	\[
	G[PF(G)],\qquad G[PFF(G)]
	\]
	\noindent also have a perfect matching and are both K\H{o}nig-Egerváry
	graphs, so the most natural thing in this context is to study the following.
	
	\begin{problem}
		Characterize the unimodular K\H{o}nig-Egerváry graphs $G$ such that $G[PF(G)]=G$. 
	\end{problem}
	
\begin{problem}
	Characterize unimodular perfect-flower-free K\H{o}nig-Egerváry graphs, that is, graphs $G$ such that $G[PFF(G)]=G$.
\end{problem}

	\medskip
	\noindent\textbf{Acknowledgment.}
	We thank the referee for the careful reading and helpful suggestions.

	\section*{Declaration of generative AI and AI-assisted technologies in the writing process}
	During the preparation of this work the authors used ChatGPT-3.5 in order to improve the grammar of several paragraphs of the text. After using this service, the authors reviewed and edited the content as needed and take full responsibility for the content of the publication.

\section*{Data availability}

Data sharing not applicable to this article as no datasets were generated or analyzed during the current study.

\section*{Declarations}

\noindent\textbf{Conflict of interest} \ The authors declare that they have no conflict of interest.

\bibliographystyle{apalike}

\bibliography{TAGcitasV2025}

\end{document}